%% file: sinum_main.tex
\documentclass{siamart190516}

\input{ex_shared}

\ifpdf
\hypersetup{
  pdftitle={Divergence-Free Nonconforming VEM For Stokes},
  pdfauthor={H. Wei, X. Huang, and A. Li}
}
\fi

\usepackage{amssymb}
\usepackage[leqno]{amsmath}
\usepackage{mathrsfs}
\usepackage{stmaryrd}
\usepackage{chemarrow}
\usepackage{enumerate}
\usepackage{graphicx}
\usepackage{subfigure}  
\usepackage{hyperref}
\usepackage[hyperpageref]{backref}
\usepackage[all]{xy}
\usepackage{tikz}
\usetikzlibrary{arrows}

\usepackage{multirow}

  \newcounter{mnote}
  \let\oldmarginpar\marginpar
    \renewcommand\marginpar[1]{\-\oldmarginpar[\raggedleft\footnotesize #1]%
    {\raggedright\footnotesize #1}}

\newtheorem{example}[theorem]{Example}

\newcommand{\dd}{\,{\rm d}}
\newcommand{\bs}{\boldsymbol}

\newcommand{\curl}{\operatorname{curl}}
\renewcommand{\div}{\operatorname{div}}

\DeclareMathOperator*{\tr}{tr}

\newcommand{\dev}{\operatorname{dev}}

\numberwithin{equation}{section}

\begin{document}

\maketitle

\begin{abstract}
Piecewise divergence-free nonconforming virtual elements are designed for Stokes problem in any dimensions.
After introducing a local energy projector based on the Stokes problem and the stabilization,
a divergence-free nonconforming virtual element method is proposed for Stokes problem.
A detailed and rigorous error analysis is presented for the discrete method.
An important property in the analysis is that the local energy projector commutes with the divergence operator. 
With the help of a divergence-free interpolation operator onto a generalized Raviart-Thomas element space, a pressure-robust nonconforming virtual element method is developed by simply modifying the right hand side of the previous discretization.
A reduced virtual element method is also discussed.
Numerical results are provided to verify the theoretical convergence.
\end{abstract}



\section{Introduction}

In this paper, we shall construct piecewise divergence-free nonconforming virtual elements for Stokes problem in any dimensions.
Assume that $\Omega\subset \mathbb{R}^d~(d\geq 2)$ is a bounded polytope.
The Stokes problem is governed by
\begin{equation}\label{eq:stokes}
\begin{cases}
-\div(\nu\boldsymbol{\varepsilon}(\bs u)) - \nabla p = \bs f & \textrm{ in } \Omega,\\
\qquad\qquad\qquad \div \bs u=0 & \textrm{ in } \Omega,\\
\qquad\qquad\qquad\quad\;\; \bs u=\bs 0 & \textrm{ on } \partial\Omega,
\end{cases}
\end{equation}
where $\bs u$ is the velocity field, $p$ is the pressure, $\boldsymbol{\varepsilon}(\bs u):=(\nabla\bs u + (\nabla\bs u)^{\intercal} )/2$ is the symmetric gradient of $\bs u$, $\bs f\in \bs L^2(\Omega;\mathbb R^d)$ is the external force field, and constant $\nu>0$ is the viscosity.
The incompressibility constraint $\div\bs u=0$ in \eqref{eq:stokes} describes the conservation of mass for the incompressible fluid.

Since the nonconforming $P_1$-$P_0$ element is a stable pair for the Stokes problem~\cite{CrouzeixRaviart1973},
as the generalization of the nonconforming $P_1$ element, it is spontaneous that the $H^1$-nonconforming virtual element in \cite{AyusodeDiosLipnikovManzini2016} is adopted to discretize the
Stokes problem in \cite{CangianiGyryaManzini2016,LiuLiChen2017}.
On the other hand, the incompressibility constraint is not satisfied exactly in general at the discrete level for the discrete methods in \cite{CangianiGyryaManzini2016,LiuLiChen2017}, 
which is very important for the Navier-Stokes problem \cite{JohnLinkeMerdonNeilanEtAl2017,CharnyiHeisterOlshanskiiRebholz2017}.
To design the discrete method with the exact divergence-free discrete velocity, one idea is to combine the discontinuous Galerkin technique and the $H(\div)$-conforming virtual elements, such as the divergence-free weak virtual element method \cite{ChenWang2019}.
The more compact idea in \cite{BeiraodaVeigaLovadinaVacca2017,BeiraodaVeigaDassiVacca2020,AntoniettiBeiraoMoraVerani2014} is to construct divergence-free conforming virtual elements in two and three dimensions by defining the space of shape functions through the local Stokes problem with Dirichlet boundary condition.
By enriching an $H(\div)$-conforming virtual element with some divergence-free functions,
a divergence-free nonconforming virtual element in two dimensions is advanced in \cite{ZhaoZhangMaoChen2019}, in which each element in the partition is required to be convex.

Following the ideas in \cite{ChenHuang2020,Huang2020}, we shall devise piecewise divergence-free $H^1$-nonconforming virtual elements in any dimensions based on the generalized Green's identity for Stokes problem, which are also $H(\div)$-nonconforming. The degrees of freedom of the proposed virtual elements for the velocity are same as those in \cite{CangianiGyryaManzini2016}, i.e. $d$ copies of the degrees of freedom of the $H^1$-nonconforming virtual elements in~\cite{AyusodeDiosLipnikovManzini2016}.
And the space of shape functions $\bs V_k(K)$ for the velocity is defined from the local Stokes problem with Neumann boundary condition, which is different from that in \cite{BeiraodaVeigaLovadinaVacca2017} due to the constraint on the boundary. Our virtual elements are locally divergence-free since $\div\bs V_k(K)=\mathbb P_{k-1}(K)$. The divergence-free velocity means the mass conservation. It is pointed out in \cite{CharnyiHeisterOlshanskiiRebholz2017} that many important conservation laws are lost with the loss of mass conservation, including energy, momentum, angular momentum. A common theme for all ‘enhanced-physics’ based schemes is that the more physics is built into the discretization, the more accurate and stable the discrete solutions are, especially over longer time intervals \cite{CharnyiHeisterOlshanskiiRebholz2017}.

A novelty of this paper is to introduce a local energy projector $\bs\Pi_k^K: \bs H^1(K; \mathbb R^d)\to\mathbb P_k(K; \mathbb R^d)$ based on the Stokes problem:
\begin{align*}
(\bs\varepsilon(\bs\Pi_k^K\bs w), \bs\varepsilon(\bs v))_K + (\div\bs v,P^K\bs w)_K &=(\bs\varepsilon(\bs w), \bs\varepsilon(\bs v))_K\quad  \forall~\bs v\in \mathbb P_k(K; \mathbb R^d),\\
(\div(\bs\Pi_k^K\bs w), q)_K&=(\div\bs w, q)_K\quad\;\quad\forall~q\in\mathbb P_{k-1}(K),
\end{align*}
while the local $H^1$ projector is adopted in all the previous papers.
The local Stokes-based projector $\bs\Pi_k^K$ commutes with the divergence operator.
Then we define a stabilization involving all the degrees of freedom of the virtual elements for the velocity except those corresponding to $\mathbb G_{k-2}(K):=\nabla\mathbb P_{k-1}(K)$.
With the help of the local projector $\bs\Pi_k^K$ and the stabilization, we propose a piecewise divergence-free nonconforming virtual element method for Stokes problem, where the velocity is discretized by the virtual elements and the pressure is discretized by the piecewise polynomials.
Differently from \cite{BeiraodaVeigaLovadinaVacca2017,CangianiGyryaManzini2016,FrerichsMerdon2020}, the computable projection $\bs\Pi_k^K\boldsymbol u_h$ in this paper is divergence-free on each element $K$.

Furthermore, applying the technique in \cite{BeiraodaVeigaLovadinaVacca2017}, we remove the degrees of freedom corresponding to $\mathbb G_{k-2}(K)$ for the velocity, reduce the space of shape functions $\bs V_k(K)$ to $\bs{\widetilde V}_k(K)=\{\bs v\in\bs V_k(K): \div\bs v\in\mathbb P_{0}(K)\}$, and then derive the reduced virtual element method, in which the pressure is discretized by piecewise constant functions. Hence we can first acquire the discrete velocity by solving the reduced discrete method, and then recover the discrete pressure elementwisely.

A detailed and rigorous error analysis is presented for the piecewise divergence-free nonconforming virtual element method.
We first prove the norm equivalence of the stabilization on the kernel of the local projector $\bs\Pi_k^K$.
Then the interpolation error estimate is acquired after setting up the Galerkin orthogonality of the interpolation operator.
With the norm equivalence of the stabilization and the interpolation error estimate, we build up the discrete inf-sup condition, and thus the piecewise divergence-free nonconforming virtual element method is well-posed. Finally the optimal error estimate comes from the discrete inf-sup condition and the interpolation error estimate in a standard way.

Following the ideas in \cite{Linke2014,JohnLinkeMerdonNeilanEtAl2017}, we devise a pressure-robust nonconforming virtual element method for the Stokes problem~\eqref{eq:stokes} by modifying the right hand side of the previous discrete method.
We first define a generalized Raviart-Thomas element space $\widetilde{\bs {RT}}_h$ based on the partition $\mathcal T_h$ by extending the Raviart-Thomas element \cite{RaviartThomas1977,Nedelec1980,ArnoldFalkWinther2006} on simplices to polytopes.
And introduce a divergence-free interpolation operator $\bs I_h^{RT}: \bs V_h\to \widetilde{\bs {RT}}_h$ satisfying
\begin{equation}\label{intro:IhRTdivfree}
\div(\bs I_h^{RT}\bs v_h)=\div_h\bs v_h\quad\forall~\bs v_h\in \bs V_h.
\end{equation}
Property \eqref{intro:IhRTdivfree} is vital to derive the pressure-robust error estimate for velocity, which is true for our divergence-free virtual element, but not the case for the virtual element in \cite{CangianiGyryaManzini2016}. Then replace $\langle\bs f, \bs v_h\rangle$ by $(\bs f, \bs I_h^{RT}\bs v_h)$ to get the pressure-robust discretization. 
Very recently a pressure-robust conforming virtual element method for Stokes problem in two dimensions is proposed in \cite{FrerichsMerdon2020} by employing a similar idea, while the computable $\bs\Pi_k^K\boldsymbol{u}_h$ in \cite{FrerichsMerdon2020} is not divergence-free.

The rest of this paper is organized as follows. In Section~2, we present some notation and inequalities. The divergence-free nonconforming virtual elements, local energy projector, stabilization and interpolation operator are constructed in Section~3.
We show the divergence-free nonconforming virtual element methods for the Stokes problem and the error analysis in Section~4. A reduced virtual element method is given in Section~5.
In Section~6, numerical results are provided to verify the theoretical convergence.

\section{Preliminaries}

\subsection{Notation}

Denote by $\mathbb{M}$ the space of all $d\times d$ tensors, $\mathbb{S}$ the space of all symmetric $d\times d$ tensors, and $\mathbb{K}$ the space of all skew-symmetric $d\times d$ tensors. Denote the deviatoric part and the trace of the tensor $\bs\tau$ by $\dev\bs\tau$ and $\tr\bs\tau$ accordingly,  then we have
\[
\dev\bs\tau=\bs\tau-\frac{1}{d}(\tr\bs\tau)\bs I.
\]
Given a bounded domain $K\subset\mathbb{R}^{d}$ and a
non-negative integer $m$, let $H^m(K)$ be the usual Sobolev space of functions
on $K$, and $\boldsymbol{H}^m(K; \mathbb{X})$ be the usual Sobolev space of functions taking values in the finite-dimensional vector space $\mathbb{X}$ for $\mathbb{X}$ being $\mathbb{M}$, $\mathbb{S}$, $\mathbb{K}$ or $\mathbb{R}^d$. The corresponding norm and semi-norm are denoted respectively by
$\Vert\cdot\Vert_{m,K}$ and $|\cdot|_{m,K}$. Let $(\cdot, \cdot)_K$ be the standard inner product on $L^2(K)$ or $\boldsymbol{L}^2(K; \mathbb{X})$. If $K$ is $\Omega$, we abbreviate
$\Vert\cdot\Vert_{m,K}$, $|\cdot|_{m,K}$ and $(\cdot, \cdot)_K$ by $\Vert\cdot\Vert_{m}$, $|\cdot|_{m}$ and $(\cdot, \cdot)$,
respectively. Let $\boldsymbol{H}_0^m(K; \mathbb{R}^d)$ be the closure of $\boldsymbol{C}_{0}^{\infty}(K; \mathbb{R}^d)$ with
respect to the norm $\Vert\cdot\Vert_{m,K}$.
For integer $k\geq0$,
notation $\mathbb P_k(K)$ stands for the set of all
polynomials over $K$ with the total degree no more than $k$. Set $\mathbb P_{-1}(K)=\mathbb P_{-2}(K)=\{0\}$.
And denote by $\mathbb P_{k}(K; \mathbb{X})$ the vectorial or tensorial version space of $\mathbb P_{k}(K)$.
Let $Q_k^{K}$ ($\bs Q_k^{K}$) be the $L^2$-orthogonal projector onto $\mathbb P_k(K)$ ($\mathbb P_{k}(K; \mathbb{X})$).

Let $\{\mathcal {T}_h\}$ be a family of partitions
of $\Omega$ into nonoverlapping simple polytopal elements with $h:=\max\limits_{K\in \mathcal {T}_h}h_K$
and $h_K:=\mbox{diam}(K)$.
Let $\mathcal{F}_h^r$ be the set of all $(d-r)$-dimensional faces
of the partition $\mathcal {T}_h$ for $r=1, 2$.
Moreover, we set for each $K\in\mathcal{T}_h$
\[
\mathcal{F}(K):=\{F\in\mathcal{F}_h^1: F\subset\partial K\}.
\]
Similarly, for $F\in\mathcal{F}_h^1$, we define
\[
\mathcal E(F):=\{e\in\mathcal{F}_h^{2}: e\subset\overline{F}\}.
\]
For any $F\in\mathcal{F}_h^1$,
denote by $h_F$ its diameter and fix a unit normal vector $\boldsymbol{n}_F$.
For any $F\subset\partial K$,
denote by $\bs n_{K, F}$ the
unit outward normal to $\partial K$. Without causing any confusion, we will abbreviate $\bs n_{K, F}$ as $\bs n$ for simplicity.

For non-negative integer $k$, let
\[
\mathbb P_k(\mathcal T_h):=\{v\in L^2(\Omega): v|_K\in \mathbb P_k(K)\textrm{ for each } K\in\mathcal T_h\}.
\]
Define
\[
\boldsymbol{H}^1(\mathcal T_h; \mathbb{R}^d)
:=\{\bs v\in \bs L^2(\Omega; \mathbb{R}^d): \bs v|_K\in \boldsymbol{H}^1(K; \mathbb{R}^d)\textrm{ for each } K\in\mathcal T_h\},
\]
and the usual broken $H^1$-type norm and semi-norm
\[
\|\bs v\|_{1,h}:=\Big (\sum_{K\in\mathcal T_h}\|\bs v\|_{1,K}^2\Big )^{1/2},\quad |\bs v|_{1,h}:=\Big (\sum_{K\in\mathcal T_h}|\bs v|_{1,K}^2\Big )^{1/2}.
\]
%
Let $\boldsymbol{\varepsilon}_h$ and $\div_h$ be the piecewise counterparts of $\boldsymbol{\varepsilon}$ and $\div$ with respect to $\mathcal T_h$.

We introduce jumps on ($d-1$)-dimensional faces.
Consider two adjacent elements $K^+$ and $K^-$ sharing an interior ($d-1$)-dimensional face $F$.
Denote by $\bs n^+$ and $\bs n^-$ the unit outward normals
to the common face $F$ of the elements $K^+$ and $K^-$, respectively.
For a scalar-valued or tensor-valued function $v$, write $v^+:=v|_{K^+}$ and $v^-
:=v|_{K^-}$.   Then define the jump on $F$ as
follows:
\[
\llbracket v\rrbracket:=v^+\bs n_{F}\cdot \bs n^++v^-\bs n_{F}\cdot \bs n^-.
\]
On a face $F$ lying on the boundary $\partial\Omega$, the above term is
defined by $\llbracket v\rrbracket
   :=v\bs n_{F}\cdot\bs n.$

Denote the space of rigid motions by
\[
\bs{RM}:=\{\bs c+\bs A\bs x: \;\bs c\in\mathbb R^d,\; \bs A\in \mathbb K\},
\]
where $\bs x:=(x_1, \cdots, x_d)^{\intercal}$.
For any $\bs v:=(v_1,\cdots, v_d)^{\intercal}\in\bs H^1(K;\mathbb R^d)$,
$\bs\curl\,\bs v\in \bs L^2(K;\mathbb K)$ is defined by
\[
(\bs\curl\,\bs v)_{ij}:=\frac{\partial v_i}{\partial x_j}-\frac{\partial v_j}{\partial x_i}\quad\textrm{ for } i,j=1,\cdots,d.
\]
For positive integer $k$, set $\mathbb G_{k-2}(K):=\nabla\mathbb P_{k-1}(K)$. Take $\mathbb G_{k-2}^{\oplus}(K)$ being any subspace of $\mathbb P_{k-2}(K;\mathbb R^d)$ such that
\begin{equation}\label{eq:Pkdirectsum}
\mathbb P_{k-2}(K; \mathbb R^d)=\mathbb G_{k-2}^{\oplus}(K)\oplus \mathbb G_{k-2}(K),
\end{equation}
where $\oplus$ is the direct sum.
One choice of $\mathbb G_{k-2}^{\oplus}(K)$ is given by (3.11) in \cite{Arnold2018,ArnoldFalkWinther2006}
\begin{equation}\label{eq:Gk-2perpchoice1}
\mathbb G_{k-2}^{\oplus}(K)=\begin{cases}
\bs x^{\perp}\mathbb P_{k-3}(K), & \textrm{ for }  d=2,\\
\bs x\wedge\mathbb P_{k-3}(K;\mathbb R^3), & \textrm{ for } d=3,
\end{cases}
\end{equation}
where $\bs x^{\perp}:=\begin{pmatrix}
x_2\\ -x_1
\end{pmatrix}$ and $\wedge$ is the exterior product.
Let $\bs Q_{\mathbb G_{k-2}^{\oplus}}^K$ be the $L^2$-orthogonal projector onto $\mathbb G_{k-2}^{\oplus}(K)$.

\subsection{Mesh conditions and some inequalities}
We impose the following conditions on the mesh $\mathcal T_h$ in this paper:
\begin{itemize}
 \item[(A1)] Each element $K\in \mathcal T_h$ is star-shaped with respect to a ball $B_K\subset K$  with radius $h_K/\gamma_K$, where the chunkiness parameter $\gamma_K$ is uniformly bounded;
 \item[(A2)] There exists a shape regular simplicial mesh $\mathcal T_h^*$ such that 
 \begin{itemize}
 \item[$-$] each $K\in \mathcal T_h$ is a union of some simplexes in $\mathcal T_h^*$;
 \item[$-$] for each $K\in \mathcal T_h$, $\mathcal T_K:=\{K'\in\mathcal T_h^*: K'\subset K\}$ is a quasi-uniform partition of $K$, and the mesh size of $\mathcal T_K$ is proportional to $h_K$. 
 \end{itemize}
\end{itemize}
Throughout this paper, we use
``$\lesssim\cdots $" to mean that ``$\leq C\cdots$", where
$C$ is a generic positive constant independent of the mesh size $h$ and the viscosity $\nu$, but may depend on the chunkiness parameter of the polytope, the degree of polynomials $k$, the dimension of space $d$, and the shape regularity and quasi-uniform constants of the virtual triangulation $\mathcal T^*_h$,
which may take different values at different appearances. And $A\eqsim B$ means $A\lesssim B$ and $B\lesssim A$.


Under the mesh condition (A1),
we have the trace inequality of $H^1(K)$ \cite[(2.18)]{BrennerSung2018}
\begin{equation}\label{L2trace}
\|v\|_{0,\partial K}^2 \lesssim h_K^{-1}\|v\|_{0,K}^2 + h_K |v|_{1,K}^2 \quad \forall~v\in H^1(K),
\end{equation}
the Poincar\'e-Friedrichs inequality \cite[(2.15)]{BrennerSung2018}
\begin{equation}\label{eq:Poincare-Friedrichs2}
\|v\|_{0,K}\lesssim h_K|v|_{1,K} + h_K^{1-d/2}\left|\int_{\partial K}v\dd s\right|\quad \forall~v\in H^1(K),
\end{equation}
and
the Korn's second inequality \cite{Duran2012}
\begin{equation}\label{eq:Korn}
|\bs v|_{1,K}\lesssim \|\bs\varepsilon(\bs v)\|_{0,K}\quad\forall~\bs v\in \bs H^1(K;\mathbb R^d) \textrm{ satisfying } \bs Q_0^{K}(\bs\curl\,\bs v)=\bs0.
\end{equation}
Recall the
Babu\v{s}ka-Aziz inequality \cite{BernardiCostabelDaugeGirault2016}: for any $q\in L^2(K)$, there exists $\bs v\in \bs H^1(K;\mathbb R^d)$ such that
\begin{equation}\label{eq:babuska-aziz}
\div\bs v=q,\quad h_K^{-1}\|\bs v\|_{0, K}+|\bs v|_{1, K}\lesssim \|q\|_{0,K}.
\end{equation}
When $q\in L_0^2(K)$, we can choose $\bs v\in \bs H_0^1(K;\mathbb R^d)$.
For any $\bs\tau\in \bs L^2(K;\mathbb M)$ satisfying $Q_0^K(\tr\bs\tau)=0$, it holds (cf. \cite[Lemma~3.4]{ChenHuHuang2018})
\begin{equation}\label{eq:tensorineqlty}
\|\bs\tau\|_{0,K}\lesssim \|\dev\bs\tau\|_{0,K}+\|\div\bs\tau\|_{-1,K}.
\end{equation}

Let $K_s\subset\mathbb R^n$ be the regular inscribed simplex of $B_K$, where all the edges of $K_s$ have the same length.
It holds for any nonnegative integers $\ell$ and $i$ that \cite[Lemma~4.3 and Lemma~4.4]{Huang2020}
\begin{equation}\label{eqn:polynomialequiv}
\|q\|_{0, K}\eqsim \|q\|_{0, K_s}\quad\forall~q\in \mathbb P_{\ell}(K),
\end{equation}
\begin{equation}\label{eq:polyinverse}
 \| q \|_{0,K} \lesssim h_K^{-i}\| q \|_{-i,K} \quad \forall~q\in \mathbb P_{\ell}(K).
\end{equation}

\begin{lemma}
For any nonnegative integers $\ell$, $i$ and $j$, we have
\begin{equation}\label{eq:polyinverseEqual}
 h_K^{-j}\| q \|_{-j,K} \eqsim h_K^{-i}\| q \|_{-i,K} \quad \forall~q\in \mathbb P_{\ell}(K).
\end{equation}
\end{lemma}
\begin{proof}
It is sufficient to prove
\begin{equation}\label{eq:201908027-1}
\| q \|_{0,K} \eqsim h_K^{-i}\| q \|_{-i,K} \quad \forall~q\in \mathbb P_{\ell}(K)
\end{equation}
with $i\geq 1$.
Applying the Poincar\'e-Friedrichs inequality~\eqref{eq:Poincare-Friedrichs2} recursively, we get for any $v\in H_0^i(K)$ that
\[
(q, v)_K\leq\|q\|_{0,K}\|v\|_{0,K}\lesssim h_K\|q\|_{0,K}|v|_{1,K}\lesssim \cdots\lesssim h_K^i\|q\|_{0,K}|v|_{i,K}.
\]
Then it follows
\[
\| q \|_{-i,K}=\sup_{v\in H_0^i(K)}\frac{(q, v)_K}{|v|_{i,K}}\lesssim h_K^i\|q\|_{0,K},
\]
which together with \eqref{eq:polyinverse} yields \eqref{eq:201908027-1}.
\end{proof}

Recall the error estimates of the $L^2$ projection.
For each $F\in\mathcal{F}(K)$ and nonnegative integer $\ell$, we have  
\begin{align}
\label{eq:QKerror}
\|v-Q_{\ell}^Kv\|_{0,K}&\lesssim h_K^{\ell+1}|v|_{\ell+1, K}\quad\forall~v\in H^{\ell+1}(K),\\
\label{eq:QFerror}
\|v-Q_{\ell}^Fv\|_{0,F}&\lesssim h_K^{\ell+1/2}|v|_{\ell+1, K}\quad\forall~v\in H^{\ell+1}(K).
\end{align}

\begin{lemma}
We have for any $q\in \mathbb P_{k-1}(K)$ that
\begin{equation}\label{eq:infsupPolynomial}
\| q \|_{0,K}\lesssim \sup_{\bs w\in\mathbb P_k(K;\mathbb R^d)}\frac{(\div\bs w, q)_K}{h_{K}^{-1}\|\bs w\|_{0, K}+|\bs w|_{1, K}}.
\end{equation}
\end{lemma}
\begin{proof}
Due to \eqref{eq:babuska-aziz},
there exists $\bs v\in \bs H^1(K_s;\mathbb R^d)$
such that
\[
\div\bs v=q|_{K_s}\quad h_{K_s}^{-1}\|\bs v\|_{0, K_s}+|\bs v|_{1, K_s}\lesssim \|q\|_{0,K_s}.
\]
Let $\bs I_{K_s}^{\rm BDM}: \bs H^1(K_s;\mathbb R^d)\to \mathbb P_k(K_s;\mathbb R^d)$ be the Brezzi-Douglas-Marini interpolation \cite{BoffiBrezziFortin2013,ArnoldFalkWinther2006},
then
\[
\div(\bs I_{K_s}^{\rm BDM}\bs v)=Q_{k-1}^{K_s}\div\bs v=q|_{K_s},
\]
\[
\|\bs v-\bs I_{K_s}^{\rm BDM}\bs v\|_{0,K_s}\lesssim h_{K_s}|\bs v|_{1,K_s}\lesssim h_{K_s}\|q\|_{0,K_s}.
\]
It follows from the inverse inequality \eqref{eq:polyinverse} and \eqref{eq:QKerror} that
\begin{align*}
|\bs I_{K_s}^{\rm BDM}\bs v|_{1, K_s}&=|\bs I_{K_s}^{\rm BDM}\bs v - \bs Q_0^{K_s}\bs v|_{1, K_s}\lesssim h_{K_s}^{-1}\|\bs I_{K_s}^{\rm BDM}\bs v - \bs Q_0^{K_s}\bs v\|_{0, K_s} \\
&\lesssim h_{K_s}^{-1}\|\bs v-\bs I_{K_s}^{\rm BDM}\bs v\|_{0, K_s} + h_{K_s}^{-1}\|\bs v - \bs Q_0^{K_s}\bs v\|_{0, K_s} \\
&\lesssim |\bs v|_{1,K_s}\lesssim \|q\|_{0,K_s}.
\end{align*}
Noting that $\bs I_{K_s}^{\rm BDM}\bs v\in\mathbb P_k(K_s;\mathbb R^d)$ can be spontaneously extended to the domain $K$, let $\bs w\in \mathbb P_k(K;\mathbb R^d)$ such that $\bs w|_{K_s}=\bs I_{K_s}^{\rm BDM}\bs v$. Thus
\[
(\div\bs w - q)|_{K_s}=0, \quad h_{K_s}^{-1}\|\bs w\|_{0, K_s}+|\bs w|_{1, K_s}\lesssim \|q\|_{0,K_s}.
\]
Again due to $\div\bs w - q$ being a polynomial,
$
(\div\bs w - q)|_{K_s}=0$  implies  $\div\bs w = q$ on $K$.
And it follows from \eqref{eqn:polynomialequiv} that
\[
h_{K}^{-1}\|\bs w\|_{0, K}+|\bs w|_{1, K}\lesssim h_{K_s}^{-1}\|\bs w\|_{0, K_s}+|\bs w|_{1, K_s}\lesssim \|q\|_{0,K_s}\leq \|q\|_{0,K}.
\]
Therefore we arrive at \eqref{eq:infsupPolynomial}.
\end{proof}

\section{Divergence-Free Nonconforming Virtual Elements}

We will construct the divergence-free nonconforming virtual elements for Stokes problem in this section.

\subsection{Virtual elements}
For any $K\in\mathcal T_h$, $\bs u, \bs v\in \bs H^1(K;\mathbb R^d)$ and $p\in L^2(K)$
satisfying $\div\boldsymbol{\varepsilon}(\bs u)+\nabla p\in \bs L^2(K;\mathbb R^d)$, and
$(\boldsymbol{\varepsilon}(\bs u)\bs n+p\bs n)|_F\in\bs L^2(F;\mathbb R^d)$ for each $F\in\mathcal F(K)$,
it follows from the integration by parts that
\begin{equation}\label{eq:StokesGreen1}
(\boldsymbol{\varepsilon}(\bs u), \boldsymbol{\varepsilon}(\bs v))_K + (\div \bs v, p)_K 
=-(\div\boldsymbol{\varepsilon}(\bs u)+\nabla p, \bs v)_K+
(\boldsymbol{\varepsilon}(\bs u)\bs n+p\bs n, \bs v)_{\partial K}. 
\end{equation}
Inspired by the Green's identity~\eqref{eq:StokesGreen1},
we propose the following local degrees of freedom of the divergence-free nonconforming virtual elements for Stokes problem
\begin{align}
(\bs v, \bs q)_F & \quad\forall~\bs q\in\mathbb P_{k-1}(F; \mathbb R^d) \textrm{ on each }  F\in\mathcal F(K), \label{dof1}\\
(\bs v, \bs q)_K & \quad\forall~\bs q\in\mathbb P_{k-2}(K; \mathbb R^d)=\mathbb G_{k-2}^{\oplus}(K)\oplus \mathbb G_{k-2}(K). \label{dof2}
\end{align}
Denote by $\mathcal N_k(K)$ all the degrees of freedom \eqref{dof1}-\eqref{dof2}.
And define the space of shape functions as
\begin{align*}
\bs V_k(K):=\{\bs v\in \bs H^1(K;\mathbb R^d): &\div\bs v\in\mathbb P_{k-1}(K), \textrm{ there exists some } \\
& s\in L^2(K) \textrm{ such that } \div\boldsymbol{\varepsilon}(\bs v)+\nabla s\in \mathbb G_{k-2}^{\oplus}(K),  \\
&\textrm{ and } (\boldsymbol{\varepsilon}(\bs v)\bs n+s\bs n)|_F\in\mathbb P_{k-1}(F;\mathbb R^d) \;\forall~F\in\mathcal F(K)\}.
\end{align*}
By the direct sum decomposition \eqref{eq:Pkdirectsum}, clearly we have $\mathbb P_{k}(K;\mathbb R^d)\subseteq\bs V_k(K)$.


\begin{lemma}\label{lem:VKdim}
The dimension of $\bs V_k(K)$ is same as the number of the degrees of freedom~\eqref{dof1}-\eqref{dof2}.
\end{lemma}
\begin{proof}
To count the dimension of $\bs V_k(K)$, we introduce the space
\begin{align*}
    \bs W_k(K):=\{&(\bs v, s)\in \bs H^1(K;\mathbb R^d)\times L^2(K):  \div\boldsymbol{\varepsilon}(\bs v)+\nabla s\in \mathbb G_{k-2}^{\oplus}(K),  \\
    &\div\bs v\in\mathbb P_{k-1}(K),  \textrm{ and } (\boldsymbol{\varepsilon}(\bs v)\bs n+s\bs n)|_F\in\mathbb P_{k-1}(F;\mathbb R^d) \;\forall~F\in\mathcal F(K)\}.
    \end{align*}
Consider the local Stokes problem with the Neumann boundary condition
\begin{equation}\label{eq:localStokes}
\begin{cases}
\qquad\qquad\quad\;\; -\div(\boldsymbol{\varepsilon}(\bs u)) - \nabla p=\bs f_1&\text{in
}K, \\
\qquad\qquad\qquad\qquad\qquad\quad\div\bs u=f_2&\text{in
}K, \\
\qquad\qquad\qquad\qquad\, \boldsymbol{\varepsilon}(\bs u)\bs n+p\bs n=\bs g_F&\text{on each }F\in\mathcal F(K),
\end{cases}
\end{equation}
where $\bs f_1\in\mathbb G_{k-2}^{\oplus}(K)$, $f_2\in\mathbb P_{k-1}(K)$, and $\bs g_F\in\mathbb P_{k-1}(F;\mathbb R^d)$.
Employing the Green's identity \eqref{eq:StokesGreen1}, we acquire
\begin{equation}\label{eq:localstokeseq1weak}
(\boldsymbol{\varepsilon}(\bs u), \boldsymbol{\varepsilon}(\bs v))_K + (\div \bs v, p)_K=(\bs f_1, \bs v)_K +
\sum_{F\in \mathcal F(K)}(\bs g_F, \bs v)_F.
\end{equation}
If taking $\bs v=\bs q\in \bs{RM}$ in \eqref{eq:localstokeseq1weak}, we have the compatibility condition
\begin{equation}\label{eq:compatibility}
(\bs f_1, \bs q)_K + \sum\limits_{F\in\mathcal F(K)}(\bs g_F, \bs q)_F=0 \quad\forall~\bs q\in\bs{RM}.
\end{equation}
Given $\bs f_1\in\mathbb G_{k-2}^{\oplus}(K)$, $f_2\in\mathbb P_{k-1}(K)$, and $\bs g_F\in\mathbb P_{k-1}(F;\mathbb R^d)$ satisfying the compatibility condition
\eqref{eq:compatibility},
due to \eqref{eq:localstokeseq1weak},
the weak formulation of the local problem~\eqref{eq:localStokes} is to find $\bs u\in \bs H^1(K;\mathbb R^d)/\bs{RM}$ and $p\in L^2(K)$ such that
\begin{equation}\label{eq:localStokesmixed}
\begin{cases}
(\boldsymbol{\varepsilon}(\bs u), \boldsymbol{\varepsilon}(\bs v))_K + (\div \bs v, p)_K=(\bs f_1, \bs v)_K +
\sum\limits_{F\in \mathcal F(K)}(\bs g_F, \bs v)_F, \\
\qquad\qquad\qquad\quad (\div \bs u, q)_K=(f_2, q)_K,
\end{cases}
\end{equation}
for all $\bs v\in \bs H^1(K;\mathbb R^d)/\bs{RM}$ and $q\in L^2(K)$.
According to the Babu{\v{s}}ka-Brezzi theory \cite{BoffiBrezziFortin2013}, the mixed formulation \eqref{eq:localStokesmixed} is uniquely solvable. Hence 
$$
\resizebox{\textwidth}{!}{$
\dim(\bs W_k(K)/(\bs{RM}\times\{0\}))=d\dim \mathbb P_{k-2}(K)+1 + d\sum\limits_{F\in\mathcal F(K)}\dim\mathbb P_{k-1}(F) - \dim\bs{RM}.$
}
$$
Furthermore, if all the data $\bs f_1$, $f_2$ and $\bs g_F$ vanish, then the set of the solution $(\bs u, p)$ of the local Stokes problem \eqref{eq:localStokes} is exactly $\bs{RM}\times\{0\}$. As a result
\[
\dim\bs W_k(K)=d\dim \mathbb P_{k-2}(K)+1+ d\sum_{F\in\mathcal F(K)}\dim\mathbb P_{k-1}(F) .
\]

Define operator $\bs R_k^K: \bs W_k(K)\to\bs V_k(K)$ as $\bs R_k^K(\bs v, s):=\bs v$. It is obvious that $\bs R_k^K\bs W_k(K)=\bs V_k(K)$. For any $(\bs v, s)\in\bs W_k(K)\cap\ker(\bs R_k^K)$, it follows $\bs v=\bs0$. By the definition of $\bs W_k(K)$,  we have $$
\nabla s\in \mathbb G_{k-2}^{\oplus}(K) \quad\textrm{ and } \;\;\; s|_F\in\mathbb P_{k-1}(F) \quad \forall~F\in\mathcal F(K).
$$
Thus $\nabla s=0$, and $s\in \mathbb P_0(K)$. This implies $\bs W_k(K)\cap\ker(\bs R_k^K)=\{0\}\times\mathbb P_0(K)$ and $\dim\bs W_k(K)\cap\ker(\bs R_k^K)=1$. Thanks to
\[
\dim\bs V_k(K)=\dim\bs R_k^K\bs W_k(K)=\dim\bs W_k(K)-\dim\bs W_k(K)\cap\ker(\bs R_k^K),
\]
we acquire $\dim\bs V_k(K)=d\dim \mathbb P_{k-2}(K) + d\sum\limits_{F\in\mathcal F(K)}\dim\mathbb P_{k-1}(F)$.
\end{proof}

Thanks to Lemma~\ref{lem:VKdim}, following the argument in \cite[Lemma 3.1]{AyusodeDiosLipnikovManzini2016} and \cite[Proposition 3.2]{BeiraodaVeigaLovadinaVacca2017},
it is easy to show that the degrees of freedom~\eqref{dof1}-\eqref{dof2} are unisolvent for the local virtual element space $\bs V_k(K)$.

The degrees of freedom~\eqref{dof1}-\eqref{dof2} are same as those in \cite{AyusodeDiosLipnikovManzini2016, CangianiGyryaManzini2016,ChenHuang2020}, but the spaces of shape functions $\bs V_k(K)$ are different.
We use the local Stokes problem with the Neumann boundary condition to define $\bs V_k(K)$,  while the local Poisson equation with the Neumann boundary condition is adopted in \cite{AyusodeDiosLipnikovManzini2016, CangianiGyryaManzini2016,ChenHuang2020}. 
The virtual elements in this paper are piecewise divergence-free. 
\begin{remark}
Assume $K$ is a simplex. It follows
\[
    \dim\bs V_k(K)-\dim\mathbb P_{k}(K;\mathbb R^d)=d(d+1)C_{k+d-2}^{d-1}+dC_{k+d-2}^{d}-dC_{k+d}^{d}=(k-1)dC_{k+d-2}^{d-2}.
\]
Hence, we have $\bs V_1(K)=\mathbb P_{1}(K;\mathbb R^d)$ for $k=1$, and the virtual element $(K$, $\mathcal N_1(K)$, $\bs V_1(K))$ is exactly the nonconforming $P_1$ element in~\cite{CrouzeixRaviart1973}. For $k\geq2$, $\mathbb P_{k}(K;\mathbb R^d)$ is a proper subset of $\bs V_k(K)$.
\end{remark}

\subsection{Local projection}
With the degrees of freedom~\eqref{dof1}-\eqref{dof2}, define a local operator $\bs\Pi_k^K: \bs H^1(K; \mathbb R^d)\to\mathbb P_k(K; \mathbb R^d)$ as follows: given $\bs w\in \bs H^1(K; \mathbb R^d)$, let $\bs \Pi_k^{K}\bs w\in\mathbb P_k(K; \mathbb R^d)$ and $P^K\bs w\in \mathbb P_{k-1}(K)$ be the solution of the local Stokes problem 
\begin{align}
(\bs\varepsilon(\bs\Pi_k^K\bs w), \bs\varepsilon(\bs v))_K + (\div\bs v,P^K\bs w)_K &=(\bs\varepsilon(\bs w), \bs\varepsilon(\bs v))_K\quad  \forall~\bs v\in \mathbb P_k(K; \mathbb R^d),\label{eq:H1projlocal1}\\
\div(\bs\Pi_k^K\bs w)&=Q_{k-1}^K(\div\bs w),\label{eq:H1projlocal2}\\
\bs Q_0^{K}(\bs\curl\,\bs\Pi_k^K\bs w)&=\bs Q_0^{K}(\bs\curl\,\bs w),\label{eq:H1projlocal3}\\
\bs Q_0^{K}(\bs \Pi_k^K\bs w)&=\bs Q_0^{K}\bs w.\label{eq:H1projlocal4}
\end{align}
Similarly as (3.2) in \cite{BrennerSung2018}, an equivalent formulation of the local Stokes problem~\eqref{eq:H1projlocal1}-\eqref{eq:H1projlocal4} is
\begin{align*}
(\!(\bs\Pi_k^K\bs w, \bs v)\!)_K + (\div\bs v,P^K\bs w)_K &=(\!(\bs w, \bs v)\!)_K\qquad  \forall~\bs v\in \mathbb P_k(K; \mathbb R^d),\\
(\div(\bs\Pi_k^K\bs w), q)_K&=(\div\bs w, q)_K\quad \forall~q\in \mathbb P_{k-1}(K),
\end{align*}
where
\[
(\!(\bs w, \bs v)\!)_K:=(\bs\varepsilon(\bs w), \bs\varepsilon(\bs v))_K + \bs Q_0^{K}(\bs\curl\,\bs w):\bs Q_0^{K}(\bs\curl\,\bs v)+\bs Q_0^{K}\bs w\cdot\bs Q_0^{K}\bs v
\]
with symbols $:$ and $\cdot$ being the inner products of the tensors and vectors respectively.

The inf-sup condition \eqref{eq:infsupPolynomial} indicates $(\mathbb P_k(K; \mathbb R^d), \mathbb P_{k-1}(K))$ is a stable pair for Stokes problem, thus
the local Stokes problem~\eqref{eq:H1projlocal1}-\eqref{eq:H1projlocal4} is uniquely solvable.
To simplify the notation, we will rewrite $\bs \Pi_k^{K}$ as $\bs\Pi^K$.
Apparently the projector $\bs\Pi^K$ can be computed using only the degrees of freedom~\eqref{dof1}-\eqref{dof2}.
The unique solvability of the local Stokes problem~\eqref{eq:H1projlocal1}-\eqref{eq:H1projlocal4} implies
the operator $\bs\Pi^K$ is a projector, i.e.
\begin{equation*}
\bs\Pi^K\bs q=\bs q\quad\forall~\bs q\in\mathbb P_k(K; \mathbb R^d).
\end{equation*}
 It follows from \eqref{eq:H1projlocal3}-\eqref{eq:H1projlocal4}, \eqref{eq:QKerror}-\eqref{eq:QFerror} and the Korn's inequality~\eqref{eq:Korn} that 
\begin{equation}\label{eq:poincare}
\|\bs v\|_{0,K} + h_K|\bs v|_{1,K} + \sum_{F\in\mathcal F(K)}h_K^{1/2}\|\bs v\|_{0,F} \lesssim h_K\|\bs\varepsilon(\bs v)\|_{0,K}\quad\forall~\bs v\in \ker(\bs\Pi^K),
\end{equation}
where $\ker(\bs\Pi^K):=\{\bs v\in \bs H^1(K;\mathbb R^d): \bs\Pi^K\bs v=\bs0\}$.
Due to \eqref{eq:H1projlocal2}, 
the local Stokes-based projector $\bs\Pi^K$ commutes with the divergence operator, i.e.
\begin{equation}\label{eq:PiKdiv}
\div(\bs v-\bs\Pi^K\bs v)=0\quad \forall~\bs v\in \bs V_k(K).
\end{equation}

By the Babu{\v{s}}ka-Brezzi theory \cite{BoffiBrezziFortin2013}, we get from the inf-sup condition \eqref{eq:infsupPolynomial} that
$$
\|\bs\varepsilon(\bs\Pi^K\bs w)\|_{0,K}+\|P^K\bs w\|_{0,K}\lesssim \sup_{\bs v\in\mathbb P_k(K; \mathbb R^d), q\in \mathbb P_{k-1}(K)}\frac{(\bs\varepsilon(\bs w), \bs\varepsilon(\bs v))_K + (\div\bs w, q)_K}{\|\bs\varepsilon(\bs v)\|_{0,K}+\|q\|_{0,K}},
$$
which means the stability
\begin{equation}\label{eq:H1projbound}
\|\bs\varepsilon(\bs\Pi^K\bs w)\|_{0,K}\lesssim \|\bs\varepsilon(\bs w)\|_{0,K}\quad \forall~\bs w\in \bs H^1(K; \mathbb R^d).
\end{equation}

\subsection{Norm equivalence}

Given $\bs w, \bs v\in\bs H^1(K;\mathbb R^d)$,
let the stabilization
\[
S_K(\bs w, \bs v):=h_K^{-2}\big(\bs Q_{\mathbb G_{k-2}^{\oplus}}^K\bs w, \bs Q_{\mathbb G_{k-2}^{\oplus}}^K\bs v\big)_K+\sum_{F\in\mathcal F(K)}h_F^{-1}(\bs Q_{k-1}^F\bs w, \bs Q_{k-1}^F\bs v)_F,
\]
and the local bilinear form
\[
a_h^K(\bs w, \bs v):=(\bs Q_{k-1}^K\bs\varepsilon(\bs w), \bs Q_{k-1}^K\bs\varepsilon(\bs v))_K+S_{K}(\bs w-\bs\Pi^{K}\bs w, \bs v-\bs\Pi^{K}\bs v).
\]
From \eqref{eq:poincare} and \eqref{eq:H1projbound}, we have for any $\bs w, \bs v\in\bs H^1(K;\mathbb R^d)$ that
\begin{equation}\label{eq:20190829-3}
a_h^K(\bs w, \bs v)\leq \left(a_h^K(\bs w, \bs w)a_h^K(\bs v, \bs v)\right)^{1/2}\lesssim \|\bs\varepsilon(\bs w)\|_{0,K}\|\bs\varepsilon(\bs v)\|_{0,K}. 
\end{equation}


Henceforth we will assume the following norm equivalence holds
\begin{equation}\label{eq:curlpolyequivalence}
h_K\|\bs\curl\,\bs q\|_{0,K}\eqsim \|\bs q\|_{0,K}\quad \forall~\bs q\in \mathbb G_{k-2}^{\oplus}(K).
\end{equation}
We first prove the norm equivalence \eqref{eq:curlpolyequivalence} for some special choices of $\mathbb G_{k-2}^{\oplus}(K)$.
\begin{lemma}
When $\mathbb G_{k-2}^{\oplus}(K)$ is the $L^2$-orthogonal complement space of $\mathbb G_{k-2}(K)$ in $\mathbb P_{k-2}(K; \mathbb R^d)$,
the norm equivalence \eqref{eq:curlpolyequivalence} holds.
\end{lemma}
\begin{proof}
Let $r \in \mathbb P_{k-1}(K)$ satisfy
\[
(\nabla r, \nabla s)_{B_K}=(\bs q, \nabla s)_{B_K}\quad \forall~s \in \mathbb P_{k-1}(B_K).
\]
Then $(\bs q-\nabla r)|_{B_K}\in\mathbb G_{k-2}^{\oplus}(B_K)$.
Since $\|\bs\curl\cdot\|_{0,B_K}$ is a norm on $\mathbb G_{k-2}^{\oplus}(B_K)$,  we get from the scaling argument that
\[
\|\bs q-\nabla r\|_{0,B_K}\lesssim h_K\|\bs\curl\,(\bs q-\nabla r)\|_{0,B_K}=h_K\|\bs\curl\,\bs q\|_{0,B_K}\leq h_K\|\bs\curl\,\bs q\|_{0,K}.
\]
Using the fact $\bs q\in \mathbb G_{k-2}^{\oplus}(K)$, we obtain from \eqref{eqn:polynomialequiv} that
\[
\|\bs q\|_{0,K}\leq \|\bs q-\nabla r\|_{0,K}\lesssim \|\bs q-\nabla r\|_{0,B_K}\lesssim h_K\|\bs\curl\,\bs q\|_{0,K}.
\]
The other side follows from the inverse inequality \eqref{eq:polyinverse}.
\end{proof}

\begin{lemma}
If $\mathbb G_{k-2}^{\oplus}(K)$ is given by \eqref{eq:Gk-2perpchoice1}, 
the norm equivalence \eqref{eq:curlpolyequivalence} holds.
\end{lemma}
\begin{proof}
We only give the proof $d=2$.
For any $q\in\mathbb P_{k-3}(K)$, noting the fact that $(\bs x^{\perp}q)|_{B_K}\in \mathbb G_{k-2}^{\oplus}(B_K)$, we achieve from the scaling argument that
\[
\|\bs x^{\perp}q\|_{0,B_K}\eqsim h_K\|\bs\curl\,(\bs x^{\perp}q)\|_{0,B_K},
\]
which combined with \eqref{eqn:polynomialequiv} implies \eqref{eq:curlpolyequivalence}.
\end{proof}

\begin{lemma}
For any $\bs v\in \bs V_k(K)$ and $s\in L^2(K)$ satisfying $\div\boldsymbol{\varepsilon}(\bs v)+\nabla s\in \mathbb G_{k-2}^{\oplus}(K)$, it holds
\begin{equation}\label{eq:inverseeq1}
h_K\|\div\boldsymbol{\varepsilon}(\bs v)+\nabla s\|_{0,K}\lesssim \|\bs\varepsilon(\bs v)\|_{0,K}.
\end{equation}
\end{lemma}
\begin{proof}
Since $\div\boldsymbol{\varepsilon}(\bs v)+\nabla s\in\mathbb G_{k-2}^{\oplus}(K)$, we get from \eqref{eq:curlpolyequivalence} and \eqref{eq:polyinverse} that
\begin{align*}
\|\div\boldsymbol{\varepsilon}(\bs v)+\nabla s\|_{0,K}&\lesssim h_K\|\bs\curl\,(\div\boldsymbol{\varepsilon}(\bs v)+\nabla s)\|_{0,K} = h_K\|\bs\curl\,(\div\boldsymbol{\varepsilon}(\bs v))\|_{0,K} \\
&\lesssim h_K^{-1}\|\bs\curl\,(\div\boldsymbol{\varepsilon}(\bs v))\|_{-2,K}\lesssim h_K^{-1}\|\bs\varepsilon(\bs v)\|_{0,K},
\end{align*}
as required.
\end{proof}

For any $F\in\mathcal{F}(K)$, let $\mathbb R_F^{d-1}$ be the $(d-1)$-dimensional affine space passing through $F$, $\mathcal{F}_F(K):=\{F'\in\mathcal{F}(K): F'\subset\mathbb R_F^{d-1}\}$, and
$
\lambda_{F}:=\bs n_{F}^{\intercal}(\boldsymbol x-\boldsymbol x_F)/h_K.
$
Clearly $\lambda_{F}|_F=0$.
Define face bubble function
$$
b_F:=\bigg(\prod_{F'\in\mathcal{F}(K)\backslash\mathcal{F}_F(K)}\lambda_{F'}\bigg)\bigg(\prod_{F'\in\mathcal{F}_F(K)}\prod_{e\in\mathcal E(F')}\bs n_{F',e}^{\intercal}\frac{\boldsymbol x-\boldsymbol x_{e}}{h_K}\bigg),
$$
for each $F\in \mathcal F(K)$. The first factor in the definition of $b_F$ is to ensure that $b_F$ vanishes on all $(d-1)$-dimensional faces of $K$ except those sharing the same affine hyperplane with $F$. And the second factor is to ensure that $b_F$ vanishes on the boundary of all $(d-1)$-dimensional faces of $K$ sharing the same affine hyperplane with $F$. Thus $b_F$ vanishes on all $(d-2)$-dimensional faces of $K$.

\begin{lemma}\label{lem:inverseeq2}
For each $F\in\mathcal F(K)$, we have for any $\bs v\in \bs V_k(K)$ and $s\in L^2(K)$ satisfying $\div\boldsymbol{\varepsilon}(\bs v)+\nabla s\in \mathbb G_{k-2}^{\oplus}(K)$ that
\begin{equation}\label{eq:inverseeq2}
\sum_{F'\in\mathcal{F}_F(K)}h_K^{1/2}\big\|\boldsymbol{\varepsilon}(\bs v)\bs n+(s-Q_0^K(s+\frac{1}{d}\div\bs v))\bs n\big\|_{0,F'}\lesssim \|\bs\varepsilon(\bs v)\|_{0,K}.
\end{equation}
\end{lemma}
\begin{proof}
Let $\bs\tau=\boldsymbol{\varepsilon}(\bs v)+(s-Q_0^K(s+\frac{1}{d}\div\bs v))\bs I$ for simplicity, then
\[
\div\bs\tau=\div\boldsymbol{\varepsilon}(\bs v)+\nabla s\in \mathbb G_{k-2}^{\oplus}(K),\quad Q_0^K(\tr\bs\tau)=0.
\]
Employing \eqref{eq:tensorineqlty}, \eqref{eq:polyinverseEqual} and \eqref{eq:inverseeq1}, we get
\begin{align}
\|\bs\tau\|_{0,K}&\lesssim \|\dev\bs\tau\|_{0,K}+\|\div\bs\tau\|_{-1,K}=\|\dev(\boldsymbol{\varepsilon}(\bs v))\|_{0,K}+\|\div\bs\tau\|_{-1,K} \notag\\
&\lesssim \|\boldsymbol{\varepsilon}(\bs v)\|_{0,K}+h_K\|\div\boldsymbol{\varepsilon}(\bs v)+\nabla s\|_{0,K}\lesssim \|\boldsymbol{\varepsilon}(\bs v)\|_{0,K}.\label{eq:20190828-1}
\end{align}
Noting that $\bs\tau\bs n|_{F'}$ is a polynomial for each $F'\in \mathcal{F}_F(K)$, let
\[
E_F(\bs\tau\bs n):=\begin{cases}
    \bs\tau\bs n|_{F'} & \textrm{ in } F'\in\mathcal{F}_F(K), \\
    \bs 0 & \textrm{ in } \mathbb R_F^{d-1}\backslash\mathcal{F}_F(K),
    \end{cases}
\]
which is a piecewise polynomial defined on $\mathbb R_F^{d-1}$.
Then we extend $E_F(\bs\tau\bs n)$ to $\mathbb R^d$. For any $\boldsymbol x\in \mathbb R^d$, let $\boldsymbol x_F^P$ be the projection of $\boldsymbol x$ on $\mathbb R_F^{d-1}$. Define
$$
E_K(\bs\tau\bs n)(\boldsymbol x):=(E_F(\bs\tau\bs n))(\boldsymbol x_F^P).
$$
Let $\mathbb R_{F'}^{d}:=\{\bs x\in\mathbb R^n:\boldsymbol x_F^P\in F'\}$, and $\bs\phi_{F}$ be a piecewise polynomial defined as
$$
\bs\phi_{F}(\bs x)=\begin{cases}
b_{F}^{2}E_K(\bs\tau\bs n), & \bs x\in\mathbb R_{F'}^{d}, F'\in\mathcal{F}_F(K), \\
\bs 0, & \bs x\in \mathbb R^{d}\backslash\bigcup\limits_{F'\in\mathcal{F}_F(K)}\mathbb R_{F'}^{d}.
\end{cases}
$$
Since $b_F$ vanishes on all $(d-2)$-dimensional faces of $K$, $\bs\phi_{F}(\bs x)$ is continuous in $\mathbb R^d$.
And we have
\begin{equation}\label{eq:20181025-2}
\|\bs\phi_{F}\|_{0,K}\lesssim \sum_{F'\in\mathcal{F}_F(K)}h_{K}^{1/2}\|\bs\tau\bs n\|_{0,F'},
\quad
\|\bs\tau\bs n\|_{0,F'}^2\eqsim(\bs\tau\bs n, \bs\phi_{F})_{F'}.
\end{equation}
Thus we obtain from \eqref{eq:20190828-1}, the inverse inequality \eqref{eq:polyinverse} and \eqref{eq:inverseeq1} that
\begin{align*}
\sum_{F'\in\mathcal{F}_F(K)}\|\bs\tau\bs n\|_{0,F'}^2&\simeq \big(\bs\tau, \bs\varepsilon(\bs \phi_F)\big)_K +(\div\boldsymbol{\varepsilon}(\bs v)+\nabla s, \bs \phi_F)_K \\
&\lesssim \|\bs\tau\|_{0,K}\|\bs\varepsilon(\bs \phi_F)\|_{0,K}+\|\div\boldsymbol{\varepsilon}(\bs v)+\nabla s\|_{0,K}\|\bs\phi_F\|_{0,K}\\
&\lesssim h_K^{-1}\|\boldsymbol{\varepsilon}(\bs v)\|_{0,K}\|\bs\phi_F\|_{0,K},
\end{align*}
which combined with \eqref{eq:20181025-2} implies \eqref{eq:inverseeq2}.
\end{proof}

With previous preparations, now we can prove the norm equivalence of the stabilization on $\ker(\bs\Pi^K)\cap \bs V_k(K)$.
\begin{lemma}
The stabilization has the norm equivalence
\begin{equation}\label{eq:normequivalence}
S_K(\bs v, \bs v)\eqsim \|\bs\varepsilon(\bs v)\|_{0,K}^2\quad\forall~\bs v\in \ker(\bs\Pi^K)\cap \bs V_k(K).
\end{equation}
\end{lemma}
\begin{proof}
Let $\bs\tau$ be defined as in the proof of Lemma~\ref{lem:inverseeq2}.
Since $\div\bs v=0$ by \eqref{eq:PiKdiv}, 
we get from \eqref{eq:inverseeq1} that
\begin{align*}
\|\bs\varepsilon(\bs v)\|_{0,K}^2&=\big(\bs\tau, \bs\varepsilon(\bs v)\big)_K =-(\div\bs\tau, \bs v)_K  + \sum_{F\in \mathcal F(K)}(\bs\tau\bs n, \bs v)_F \\
&\leq \|\div\bs\tau\|_{0,K}\big\|\bs Q_{\mathbb G_{k-2}^{\oplus}}^K\bs v\big\|_{0,K} + \sum_{F\in \mathcal F(K)}\big\|\bs \tau\bs n\big\|_{0,F}\|\bs Q_{k-1}^F\bs v\|_{0,F} \\
&\leq h_{K}^{-1}\big\|\bs Q_{\mathbb G_{k-2}^{\oplus}}^K\bs v\big\|_{0,K}\|\bs\varepsilon(\bs v)\|_{0,K} + \sum_{F\in \mathcal F(K)}\big\|\bs \tau\bs n\big\|_{0,F}\|\bs Q_{k-1}^F\bs v\|_{0,F},
\end{align*}
which together with \eqref{eq:inverseeq2} implies
$
\|\bs\varepsilon(\bs v)\|_{0,K}^2\lesssim S_K(\bs v, \bs v).
$

On the other hand, by the trace inequality~\eqref{L2trace} and \eqref{eq:poincare},
\begin{align*}
S_K(\bs v, \bs v)&=h_K^{-2}\big\|\bs Q_{\mathbb G_{k-2}^{\oplus}}^K\bs v\big\|_{0,K}^2+\sum_{F\in\mathcal F(K)}h_F^{-1}\|\bs Q_{k-1}^F\bs v\|_{0,F}^2 \\
&\lesssim h_K^{-2}\|\bs v\|_{0,K}^2+\sum_{F\in\mathcal F(K)}h_F^{-1}\|\bs v\|_{0,F}^2\lesssim h_K^{-2}\|\bs v\|_{0,K}^2+|\bs v|_{1,K}^2\lesssim \|\bs\varepsilon(\bs v)\|_{0,K}^2,
\end{align*}
which ends the proof.
\end{proof}

Thanks to \eqref{eq:H1projbound}, apparently we have
\[
\|\bs\varepsilon(\bs v)\|_{0,K}^2\eqsim\|\bs Q_{k-1}^K\bs\varepsilon(\bs v)\|_{0,K}^2+ \|\bs\varepsilon(\bs v-\bs \Pi^K \bs v)\|_{0,K}^2\quad \forall~\bs v\in \bs V_k(K),
\]
which combined with \eqref{eq:normequivalence} implies the norm equivalence
\begin{equation}\label{eq:normequivalenceH1}
a_h^K(\bs v, \bs v)\eqsim\|\bs\varepsilon(\bs v)\|_{0,K}^2\quad \forall~\bs v\in \bs V_k(K). 
\end{equation}


\subsection{Interpolation operator}

Let $\bs I_K: \bs H^1(K;\mathbb R^d)\to \bs V_k(K)$ be the canonical interpolation operator based on the degrees of freedom~\eqref{dof1}-\eqref{dof2}.
Since all the values of the degrees of freedom~\eqref{dof1}-\eqref{dof2} of $\bs v-\bs I_K\bs v$ vanish,
we have for any $\bs v\in \bs H^1(K;\mathbb R^d)$
\begin{align}\label{eq:20181012-1}
\bs\Pi^K(\bs v-\bs I_K\bs v)&=\bs0,
\\
\label{eq:20190829-2}
\div(\bs I_K\bs v)&=Q_{k-1}^K(\div\bs v).
\end{align}
Then adopting the argument in \cite[Lemma 5.1]{ChenHuang2020}, we get the Galerkin orthogonality
\begin{equation}\label{eq:20181012-2}
a_h^K(\bs v-\bs I_K\bs v, \bs w)=0\quad\forall~\bs v, \bs w \in \bs H^1(K;\mathbb R^d).
\end{equation}


Now we present the interpolation error estimate by the aid of the Galerkin orthogonality~\eqref{eq:20181012-2}.
\begin{proposition}\label{lem:IKerror}
For any $\bs v\in \bs H^{s}(K;\mathbb R^d)$ with positive integer $s\leq k+1$, we have
\begin{equation}\label{eq:IKerror}
\|\bs v-\bs I_{K}\bs v\|_{0,K}+h_K|\bs v-\bs I_{K}\bs v|_{1,K}\lesssim h_K^{s}|\bs v|_{s, K}.
\end{equation}
\end{proposition}
\begin{proof}
Take any $\bs q\in\mathbb P_k(K; \mathbb R^d)$. We obtain from \eqref{eq:normequivalenceH1}, \eqref{eq:20181012-2} with $\bs w=\bs q-\bs I_{K}\bs v$ and \eqref{eq:20190829-3} that
\begin{align*}
\|\bs\varepsilon(\bs q-\bs I_{K}\bs v)\|_{0,K}^2&\lesssim a_h^K(\bs q-\bs I_{K}\bs v, \bs q-\bs I_{K}\bs v)=a_h^K(\bs q-\bs v, \bs q-\bs I_{K}\bs v)\\
&\lesssim \|\bs\varepsilon(\bs v-\bs q)\|_{0,K}\|\bs\varepsilon(\bs q-\bs I_{K}\bs v)\|_{0,K}.
\end{align*}
Thus
\[
\|\bs\varepsilon(\bs q-\bs I_{K}\bs v)\|_{0,K}\lesssim \|\bs\varepsilon(\bs v-\bs q)\|_{0,K},
\]
and then
\[
\|\bs\varepsilon(\bs v-\bs I_{K}\bs v)\|_{0,K}\leq \|\bs\varepsilon(\bs v-\bs q)\|_{0,K}+\|\bs\varepsilon(\bs q-\bs I_{K}\bs v)\|_{0,K}\lesssim \|\bs\varepsilon(\bs v-\bs q)\|_{0,K}.
\]
By the Bramble-Hilbert Lemma \cite[Lemma~4.3.8]{BrennerScott2008}, we get
\begin{equation*}
\|\bs\varepsilon(\bs v-\bs I_{K}\bs v)\|_{0,K}\lesssim \inf_{\bs q\in\mathbb P_k(K; \mathbb R^d)}\|\bs\varepsilon(\bs v-\bs q)\|_{0,K}\lesssim h_K^{s-1}|\bs v|_{s, K}.
\end{equation*}
Finally we conclude \eqref{eq:IKerror} from \eqref{eq:poincare} and \eqref{eq:20181012-1}.
\end{proof}


\section{Divergence-Free Nonconforming Virtual Element Methods}

We will present the divergence-free nonconforming virtual element methods for the Stokes problem~\eqref{eq:stokes} in this section.
The variational formulation of the Stokes problem~\eqref{eq:stokes} is to find $\bs u\in \bs H_0^1(\Omega;\mathbb R^d)$ and $p\in L_0^2(\Omega)$ such that
\begin{align}
\nu(\boldsymbol{\varepsilon}(\bs u), \boldsymbol{\varepsilon}(\bs v)) + (\div\bs v,p)& = (\bs f, \bs v) \qquad \forall~\bs v\in \bs H_0^1(\Omega;\mathbb R^d),\label{eq:stokesvariational1}\\
(\div\bs u, q)&=0 \qquad\qquad \forall~q\in L_0^2(\Omega). \label{eq:stokesvariational2}
\end{align}

\subsection{Discretization}
Define the global virtual element space for the velocity as
\begin{align*}
\bs V_h:=\{\bs v_h\in \bs L^2(\Omega;\mathbb R^d): &\, \bs v_h|_K\in \bs V_k(K)\textrm{ for each } K\in\mathcal T_h; \; \bs Q_{k-1}^F\bs v_h \textrm{ is continuous } \\
&\quad\quad \textrm{through } F \textrm{ for all } F\in\mathcal F_h^{1};\;\bs Q_{k-1}^F\bs v_h=\bs 0 \; \textrm{ if } F\subset\partial\Omega \}.
\end{align*}
And the discrete space for the pressure is given by
\[
\mathcal Q_h=\{q_h\in L_0^2(\Omega): q_h|_K\in \mathbb P_{k-1}(K)\textrm{ for each } K\in\mathcal T_h\}.
\]
Since we use the symmetric gradient in the Stokes problem \eqref{eq:stokesvariational1}-\eqref{eq:stokesvariational2} and the discrete Korn's inequality does not hold for $\bs V_h$ when $k=1$ \cite{Brenner2004a},
hereafter we always assume integer $k\geq2$. We refer to \cite{HansboLarson2003} for overcoming the failure of the discrete Korn's inequality for the case $k=1$ by adding a jump penalization.

By the definition of $\bs V_h$, we have
\begin{equation*}
\bs Q_{k-1}^F(\llbracket\bs v_h\rrbracket)=\bs 0\quad\forall~\bs v_h\in\bs V_h, \, F\in\mathcal F_h^1.
\end{equation*}
Thanks to (3.6) in \cite{ChenHuHuang2018}, it follows
\[
|\bs v_h|_{1,h}\lesssim \|\bs\varepsilon_h(\bs v_h)\|_{0} \quad\forall~\bs v_h\in\bs V_h.
\]
Then
similarly as Lemma 4.6 and Lemma 4.8 in \cite{ChenHuang2020}, we get
for any $\bs v_h\in \bs V_h$ that 
\begin{equation*}
\sum_{F\in\mathcal F_h^{1}}h_F^{-1}\big\|\llbracket \bs v_h\rrbracket\big\|_{0,F}^2
\lesssim \|\bs\varepsilon_h(\bs v_h)\|_{0}^2,
\end{equation*}
and the discrete Poincar\'e inequality
\begin{equation}\label{eq:poincareinequality}
\|\bs v_h\|_{1,h}\lesssim |\bs v_h|_{1,h}\lesssim \|\bs\varepsilon_h(\bs v_h)\|_{0}.
\end{equation}

Let $Q_h^l: L^2(\Omega)\to \mathbb P_{l}(\mathcal T_h)$ be the $L^2$-orthogonal projector onto $\mathbb P_{l}(\mathcal T_h)$: for any $v\in L^2(\Omega)$,
\[
(Q_h^lv)|_K:= Q_{l}^{K}(v|_K)\quad\forall~K\in\mathcal T_h.
\]
The vectorial or tensorial version of $Q_h^l$ is denoted by $\bs Q_h^l$. And define $\bs\Pi_h$ as the global version of $\bs\Pi^K$ similarly.

The divergence-free nonconforming virtual element method based on the variational formulation~\eqref{eq:stokesvariational1}-\eqref{eq:stokesvariational2} for the Stokes problem~\eqref{eq:stokes} is to find $\bs u_h\in \bs V_h$ and $p_h\in \mathcal Q_h$ such that
\begin{align}
\nu a_h(\bs u_h, \bs v_h) + b_h(\bs v_h, p_h)& = \langle\bs f, \bs v_h\rangle \qquad \forall~\bs v_h\in \bs V_h,\label{eq:mixedvem1}\\
b_h(\bs u_h, q_h)&=0 \qquad\qquad\;\; \forall~q_h\in \mathcal Q_h, \label{eq:mixedvem2}
\end{align}
where
\[
a_h(\bs u_h, \bs v_h):=\sum_{K\in\mathcal T_h}a_h^K(\bs u_h, \bs v_h),
\quad
b_h(\bs v_h, p_h):=(\div_h\bs v_h, p_h),
\]
\[
\langle\bs f, \bs v_h\rangle:=\begin{cases}
(\bs f, \bs\Pi_h\bs v_h),  &  k=2,\\
(\bs f, \bs Q_h^{k-2}\bs v_h),  &  k\geq3.
\end{cases}
\]
Obviously we have from \eqref{eq:20190829-3} that 
\begin{align*}
a_h(\bs w, \bs v)&\lesssim \|\bs\varepsilon_h(\bs w)\|_0\|\bs\varepsilon_h(\bs v)\|_0\quad\,\forall~\bs w, \bs v\in \boldsymbol{H}^1(\mathcal T_h; \mathbb{R}^d),
\\
b_h(\bs v, p)&\lesssim \|\bs\varepsilon_h(\bs v)\|_0\|p\|_0 \qquad\quad\; \forall~\bs v\in \boldsymbol{H}^1(\mathcal T_h; \mathbb{R}^d), \, p\in L^2(\Omega).
\end{align*}

\subsection{Inf-sup conditions}\label{section:discreteinf-sup}
To show the well-posedness of the nonconforming virtual element method~\eqref{eq:mixedvem1}-\eqref{eq:mixedvem2}, we derive some stability results.

Denote by  $\bs I_h: \bs H_0^1(\Omega;\mathbb R^d)\to \bs V_h$ the global canonical interpolation operator based on the degrees of freedom~\eqref{dof1}-\eqref{dof2}, i.e., $(\bs I_h\bs v)|_K:=\bs I_K(\bs v|_K)$ for any $\bs v\in \bs H_0^1(\Omega;\mathbb R^d)$ and $K\in\mathcal T_h$.
Due to \eqref{eq:babuska-aziz}, \eqref{eq:20190829-2} and \eqref{eq:IKerror}, we have
$\div_h\bs V_h=\mathcal Q_h$ and the inf-sup condition (cf. \cite[Section 5.4.3]{BoffiBrezziFortin2013})
\begin{equation}\label{eq:infsup1}
\| q_h \|_{0}\lesssim \sup_{\bs v_h\in\bs V_h}\frac{b_h(\bs v_h, q_h)}{\|\bs v_h\|_{1,h}}\quad\forall~q_h\in \mathcal Q_h.
\end{equation}


\begin{lemma}\label{lem:stability}
We have the inf-sup condition
\begin{align}
&\;\nu^{1/2}\|\bs\varepsilon_h(\widetilde{\bs u}_h)\|_{0}+\nu^{-1/2}\|\widetilde{p}_h \|_{0}  \notag\\
\lesssim &\sup_{\bs v_h\in\bs V_h, q_h\in\mathcal Q_h}\frac{\nu a_h(\widetilde{\bs u}_h, \bs v_h) + b_h(\bs v_h, \widetilde{p}_h)+b_h(\widetilde{\bs u}_h, q_h)}{\nu^{1/2}\|\bs\varepsilon_h(\bs v_h)\|_{0}+\nu^{-1/2}\|q_h \|_{0}} \label{eq:infsup2}
\end{align}
for any $\widetilde{\bs u}_h\in\bs V_h$ and $\widetilde{p}_h\in\mathcal Q_h$.
\end{lemma}
\begin{proof}
By \eqref{eq:infsup1}, we have the inf-sup condition
\[
\nu^{-1/2}\| q_h \|_{0}\lesssim \sup_{\bs v_h\in\bs V_h}\frac{b_h(\bs v_h, q_h)}{\nu^{1/2}\|\bs\varepsilon_h(\bs v_h)\|_0}\quad\forall~q_h\in \mathcal Q_h.
\]
And we get from \eqref{eq:normequivalenceH1} that
\[
\nu \|\bs\varepsilon_h(\bs v_h)\|_{0}^2\lesssim \nu a_h(\bs v_h, \bs v_h)\quad\forall~\bs v_h\in\bs V_h.
\]
Therefore \eqref{eq:infsup2} follows from the Babu{\v{s}}ka-Brezzi theory.
\end{proof}

According to the stability result \eqref{eq:infsup2}, the divergence-free nonconforming virtual element method~\eqref{eq:mixedvem1}-\eqref{eq:mixedvem2} is uniquely solvable.
Thanks to \eqref{eq:H1projlocal2}, the computable $\bs\Pi^K\boldsymbol{u}_h$ is divergence-free for each $K\in\mathcal T_h$.

\subsection{Error analysis}

Now it's ready to show the optimal error estimate of the nonconforming virtual element method~\eqref{eq:mixedvem1}-\eqref{eq:mixedvem2}.
\begin{theorem}\label{thm:energyerror}
Let $(\bs u_h, p_h)\in\bs V_h\times\mathcal Q_h$ be the solution of the divergence-free nonconforming virtual element method~\eqref{eq:mixedvem1}-\eqref{eq:mixedvem2}.
Assume $\bs u\in \bs H^{k+1}(\Omega;\mathbb R^d)$, $p\in H^k(\Omega)$ and $\bs f\in \bs H^{k-1}(\Omega;\mathbb R^d)$.
Then it holds
\begin{equation}\label{eq:energyerror}
\nu\|\bs\varepsilon_h(\bs u-\bs u_h)\|_{0}+\|p-p_h\|_{0}\lesssim h^{k}(\nu|\bs u|_{k+1}+|p|_{k}+|\bs f|_{k-1}).
\end{equation}
\end{theorem}
\begin{proof}
Take any $\bs v_h\in\bs V_h$ and $q_h\in\mathcal Q_h$.
Following the arguments in \cite{ChenHuang2020,AyusodeDiosLipnikovManzini2016}, we achieve the consistency errors
\begin{equation*}
\nu(\boldsymbol{\varepsilon}(\bs u), \boldsymbol{\varepsilon}_h(\bs v_h))+(\div_h\bs v_h,p)-\langle\bs f, \bs v_h\rangle
\lesssim h^{k}(\nu|\bs u|_{k+1}+|p|_{k}+|\bs f|_{k-1})\|\bs\varepsilon_h(\bs v_h)\|_{0},
\end{equation*}
\begin{equation}\label{eq:20180830-5}
a_h(\bs I_h\bs u, \bs v_h)-(\boldsymbol{\varepsilon}(\bs u), \boldsymbol{\varepsilon}_h(\bs v_h))\lesssim h^{k}|\bs u|_{k+1}\|\bs\varepsilon_h(\bs v_h)\|_{0}.
\end{equation}
Then
we get from \eqref{eq:mixedvem1}-\eqref{eq:mixedvem2}, \eqref{eq:20190829-2} and the second equation in problem \eqref{eq:stokes} that 
\begin{align*}
&\nu a_h(\bs I_h\bs u-\bs u_h, \bs v_h) + b_h(\bs v_h, Q_h^{k-1}p-p_h)+b_h(\bs I_h\bs u-\bs u_h, q_h) \\
=&\nu a_h(\bs I_h\bs u, \bs v_h) + b_h(\bs v_h, Q_h^{k-1}p)+b_h(\bs I_h\bs u, q_h)-\langle\bs f, \bs v_h\rangle \\
=&\nu a_h(\bs I_h\bs u, \bs v_h) + b_h(\bs v_h, p)-\langle\bs f, \bs v_h\rangle \\
\lesssim&h^{k}(\nu|\bs u|_{k+1}+|p|_{k}+|\bs f|_{k-1})\|\bs\varepsilon_h(\bs v_h)\|_{0}.
\end{align*}
Due to \eqref{eq:infsup2} with $\widetilde{\bs u}_h=\bs I_h\bs u-\bs u_h$ and $\widetilde{p}_h=Q_h^{k-1}p-p_h$, it follows
\begin{align*}
&\;\nu^{1/2}\|\bs\varepsilon_h(\bs I_h\bs u-\bs u_h)\|_{0}+\nu^{-1/2}\|Q_h^{k-1}p-p_h\|_{0}  \\
\lesssim &\sup_{\bs v_h\in\bs V_h, q_h\in\mathcal Q_h}\frac{h^{k}(\nu|\bs u|_{k+1}+|p|_{k}+|\bs f|_{k-1})\|\bs\varepsilon_h(\bs v_h)\|_{0}}{\nu^{1/2}\|\bs\varepsilon_h(\bs v_h)\|_{0}+\nu^{-1/2}\|q_h \|_{0}} \\
\lesssim & \;h^{k}(\nu^{1/2}|\bs u|_{k+1}+\nu^{-1/2}|p|_{k}+\nu^{-1/2}|\bs f|_{k-1}).
\end{align*}
Hence
\[
\nu\|\bs\varepsilon_h(\bs I_h\bs u-\bs u_h)\|_{0}+\|Q_h^{k-1}p-p_h\|_{0}\lesssim h^{k}(\nu|\bs u|_{k+1}+|p|_{k}+|\bs f|_{k-1}).
\]
Thus we achieve \eqref{eq:energyerror} from the triangle inequality, \eqref{eq:IKerror} and \eqref{eq:QKerror}.
\end{proof}

\subsection{Pressure-robust discretization}

Following the ideas in \cite{Linke2014,JohnLinkeMerdonNeilanEtAl2017}, we will modify the right hand side of \eqref{eq:mixedvem1} to develop a pressure-robust nonconforming virtual element method for the Stokes problem~\eqref{eq:stokes} in this subsection.

To this end, we first extend the Raviart-Thomas element \cite{RaviartThomas1977,Nedelec1980,ArnoldFalkWinther2006} on simplices to polytopes.
For each simplex $K'\in \mathcal T_h^*$, introduce the shape function space of Raviart-Thomas element 
$\bs {RT}_{k-1}(K'):=\mathbb P_{k-1}(K';\mathbb R^d)+\bs x\mathbb P_{k-1}(K')$.
For each polytope $K\in \mathcal T_h$, let the space of shape functions
\begin{align*}
\widetilde{\bs {RT}}_{k-1}(K):=\{\bs v\in\bs H(\div, K): &\,\bs v|_{K'}\in \bs {RT}_{k-1}(K') \textrm{ for each } K'\in\mathcal T_K, \\
&\qquad\qquad\qquad\quad\textrm{ and } \div\bs v\in\mathbb P_{k-1}(K)\}.
\end{align*}
It is obvious that $\widetilde{\bs {RT}}_{k-1}(K)=\bs {RT}_{k-1}(K)$ when $K$ is a simplex.
Since the divergence operator $\div: \bs x\mathbb P_{k-1}(K)\to\mathbb P_{k-1}(K)$ is bijective, it holds the decomposition
$$
\widetilde{\bs {RT}}_{k-1}(K)=\widetilde{\bs {RT}}_{k-1}(K;\div0)\oplus \bs x\mathbb P_{k-1}(K),
$$
where $\widetilde{\bs {RT}}_{k-1}(K;\div0):=\{\bs v\in\widetilde{\bs {RT}}_{k-1}(K): \div\bs v=0\}$. Thus
\begin{align*}
    \dim\widetilde{\bs {RT}}_{k-1}(K)=&\,\#\mathcal F(\mathcal T_K)\dim\mathbb P_{k-1}(F)  + \#\mathcal T_K\dim\mathbb P_{k-2}(K;\mathbb R^d) \\
    &-\#\mathcal T_K\dim\mathbb P_{k-1}(K)+ \dim\mathbb P_{k-1}(K)\\
    =&\,\#\mathcal F(\mathcal T_K)\dim\mathbb P_{k-1}(F)  + \dim\mathbb P_{k-2}(K;\mathbb R^d) \\
    &+ (\#\mathcal T_K-1)(\dim\mathbb G_{k-2}^{\oplus}(K)-1), 
\end{align*}
where $\mathcal F(\mathcal T_K)$ is the set of all $(d-1)$-dimensional faces of the partition $\mathcal T_K$, and $F$ is some face in $\mathcal F(\mathcal T_K)$.
\begin{remark}\label{remark:RTdiv0}\rm
The local space $\widetilde{\bs {RT}}_{k-1}(K;\div0)$ can be explicitly expressed by using the finite element de Rham complex \cite{ArnoldFalkWinther2006}. Indeed we have $\widetilde{\bs {RT}}_{k-1}(K;\div0)=\curl V^{c}_{k}(\mathcal T_K)$ in two and three dimensions, where 
$$
V^{c}_{k}(\mathcal T_K):=\{v\in H^1(K): v|_{K'}\in\mathbb P_k(K')\textrm{ for each } K'\in\mathcal T_K\}
$$
in two dimensions, and
$$
V^{c}_{k}(\mathcal T_K):=\{\bs v\in H(\curl; K): \bs v|_{K'}\in\mathbb P_k(K';\mathbb R^3)\textrm{ for each } K'\in\mathcal T_K\}
$$
in three dimensions.
\end{remark}

The degrees of freedom for space $\widetilde{\bs {RT}}_{k-1}(K)$ are given by
\begin{align}
    (\bs v\cdot\bs n, q)_F & \quad\forall~q\in\mathbb P_{k-1}(F) \textrm{ on each }  F\in\mathcal F^{\partial}(\mathcal T_K), \label{rtpolytopedof1}\\
    (\bs v, \bs q)_K & \quad\forall~\bs q\in\mathbb P_{k-2}(K; \mathbb R^d), \label{rtpolytopedof2} \\
    (\bs v, \bs q)_K & \quad\forall~\bs q\in \bs B_{k-2}(\mathcal T_K), \label{rtpolytopedof3}
\end{align}
where $\mathcal F^{\partial}(\mathcal T_K):=\{F\in\mathcal F(\mathcal T_K): F\subset\partial K\}$, and
$$
\bs B_{k-2}(\mathcal T_K):=\{\bs v\in\widetilde{\bs {RT}}_{k-1}(K)\cap H_0(\div, K): (\bs v, \bs q)_K=0 \quad\forall~\bs q\in\mathbb P_{k-2}(K; \mathbb R^d)\}.
$$
The degrees of freedom \eqref{rtpolytopedof3} will disappear when $K$ is a simplex.
Clearly the degrees of freedom \eqref{rtpolytopedof1}-\eqref{rtpolytopedof3} are unisolvent for space $\widetilde{\bs {RT}}_{k-1}(K)$.
\begin{remark}\rm
Due to Remark \ref{remark:RTdiv0}, we have $\bs B_{k-2}(\mathcal T_K)=\curl\mathring{V}^{c}_{k}(\mathcal T_K)$ in two and three dimensions, where 
    $$
    \mathring{V}^{c}_{k}(\mathcal T_K):=\{v\in V^{c}_{k}(\mathcal T_K)\cap H_0^1(K): (v, q)_K=0\quad\forall~q\in\mathbb P_{k-3}(K)\}
    $$
    in two dimensions, and
    $$
    \mathring{V}^{c}_{k}(\mathcal T_K):=\{\bs v\in V^{c}_{k}(\mathcal T_K)\cap H_0(\curl; K): (\bs v, \bs q)_K=0\quad\forall~\bs q\in\curl\mathbb P_{k-2}(K;\mathbb R^3)\}
    $$
    in three dimensions.
\end{remark}

Next we introduce an interpolation operator.
Let $\bs I_K^{RT}: \bs H^1(K;\mathbb R^d)\to \widetilde{\bs {RT}}_{k-1}(K)$ be determined by
\begin{align}
    ((\bs I_K^{RT}\bs v)\cdot\bs n)|_F &=Q_{k-1}^F(\bs v\cdot\bs n) \quad \forall~F\in\mathcal F(K),\notag\\
    (\bs I_K^{RT}\bs v, \bs q)_K&=(\bs v, \bs q)_K\qquad\quad\forall~\bs q\in\mathbb P_{k-2}(K; \mathbb R^d), \notag\\
    (\bs I_K^{RT}\bs v, \bs q)_K&=(\bs\Pi^K\bs v, \bs q)_K\quad\;\,\forall~\bs q\in\bs B_{k-2}(\mathcal T_K). \label{eq:IKRT3}
\end{align}
Differently from $\bs I_K\bs v$, the projector $\bs I_K^{RT}\bs v$ can be computed using only the degrees of freedom~\eqref{dof1}-\eqref{dof2}.
And we have
$$
\bs I_K^{RT}\bs q=\bs q\quad\forall~\bs q\in\mathbb P_{k-1}(K;\mathbb R^d),
$$
\begin{equation}\label{eq:IKRTdivfree}
\div(\bs I_K^{RT}\bs v)=\bs Q_{k-1}^K(\div\bs v)\quad\forall~\bs v\in\bs H^1(K;\mathbb R^d),
\end{equation}
\begin{equation}\label{eq:20201210-2}
\|\bs v -\bs I_K^{RT}\bs v\|_{0,K}+h_K|\bs v -\bs I_K^{RT}\bs v|_{1,K}\lesssim h_K|\bs v|_{1,k}\quad\forall~\bs v\in\bs H^1(K;\mathbb R^d).
\end{equation}

Define a global generalized Raviart-Thomas element space based on the partition $\mathcal T_h$ as
$$
\widetilde{\bs {RT}}_h:=\{\bs v_h\in H_0(\div,\Omega): \bs v_h|_K\in\widetilde{\bs {RT}}_{k-1}(K) \textrm{ for each } K\in\mathcal T_h\}.
$$
And let an interpolation operator $\bs I_h^{RT}: \bs V_h\to \widetilde{\bs {RT}}_h$ be determined by $(\bs I_h^{RT}\bs v_h)|_K:=\bs I_K^{RT}(\bs v_h|_K)$ for each $K\in\mathcal T_h$. It follows from \eqref{eq:IKRTdivfree} and the fact $\div_h\bs v_h\in\mathcal Q_h$ that
\begin{equation}\label{eq:IhRTdivfree}
\div(\bs I_h^{RT}\bs v_h)=\div_h\bs v_h\quad\forall~\bs v_h\in \bs V_h.
\end{equation}

Now we revise the right hand side of \eqref{eq:mixedvem1} with the help of $\bs I_h^{RT}$ to acquire a pressure-robust virtual element method for the Stokes problem~\eqref{eq:stokes}: find $\bs u_h\in \bs V_h$ and $p_h\in \mathcal Q_h$ such that
\begin{align}
\nu a_h(\bs u_h, \bs v_h) + b_h(\bs v_h, p_h)& = (\bs f, \bs I_h^{RT}\bs v_h) \qquad \forall~\bs v_h\in \bs V_h,\label{eq:mixedvempressrobust1}\\
b_h(\bs u_h, q_h)&=0 \qquad\qquad\qquad\;\, \forall~q_h\in \mathcal Q_h. \label{eq:mixedvempressrobust2}
\end{align}

\begin{theorem}\label{thm:pressrobustenergyerror}
Let $(\bs u_h, p_h)\in\bs V_h\times\mathcal Q_h$ be the solution of the pressure-robust nonconforming virtual element method~\eqref{eq:mixedvempressrobust1}-\eqref{eq:mixedvempressrobust2}.
Assume $\bs u\in \bs H^{k+1}(\Omega;\mathbb R^d)$.
Then 
\begin{equation}\label{eq:pressrobustenergyerror}
\nu\|\bs\varepsilon_h(\bs u-\bs u_h)\|_{0}+\|Q_h^{k-1}p-p_h\|_{0}\lesssim \nu h^{k}|\bs u|_{k+1}.
\end{equation}
\end{theorem}
\begin{proof}
Noting that $\bs I_h^{RT}\bs v_h\in H_0(\div,\Omega)$,  it follows from the first equation of the Stokes problem~\eqref{eq:stokes} and \eqref{eq:IhRTdivfree} that
\begin{align*}
(\bs f, \bs I_h^{RT}\bs v_h)&=-\nu(\div(\boldsymbol{\varepsilon}(\bs u)), \bs I_h^{RT}\bs v_h)+(p,\div(\bs I_h^{RT}\bs v_h)) \\
&=-\nu(\div(\boldsymbol{\varepsilon}(\bs u)), \bs I_h^{RT}\bs v_h)+(p,\div_h\bs v_h).
\end{align*}
Then we obtain from the integration by parts and \eqref{eq:20201210-2} that
\begin{align*}
&\nu(\boldsymbol{\varepsilon}(\bs u), \boldsymbol{\varepsilon}_h(\bs v_h))+(\div_h\bs v_h,p)-(\bs f, \bs I_h^{RT}\bs v_h)\\
=&\nu(\boldsymbol{\varepsilon}(\bs u), \boldsymbol{\varepsilon}_h(\bs v_h))+\nu(\div(\boldsymbol{\varepsilon}(\bs u)), \bs I_h^{RT}\bs v_h)\\
=&\nu(\div(\boldsymbol{\varepsilon}(\bs u)), \bs I_h^{RT}\bs v_h-\bs v_h)+\nu\sum_{F\in\mathcal F_h^1}(\boldsymbol{\varepsilon}(\bs u)\bs n_{F}, \llbracket\bs v_h\rrbracket)_F\\
=&\nu(\div(\boldsymbol{\varepsilon}(\bs u))- \bs Q_h^{k-2}\div(\boldsymbol{\varepsilon}(\bs u)), \bs I_h^{RT}\bs v_h-\bs v_h) \\
& +\nu\sum_{F\in\mathcal F_h^1}(\boldsymbol{\varepsilon}(\bs u)\bs n_{F}-\bs Q_{k-1}^F(\boldsymbol{\varepsilon}(\bs u)\bs n_{F}), \llbracket\bs v_h\rrbracket)_F,
\end{align*}
which together with \eqref{eq:20201210-2} and the discrete Poincar\'e inequality \eqref{eq:poincareinequality} gives
$$
\nu(\boldsymbol{\varepsilon}(\bs u), \boldsymbol{\varepsilon}_h(\bs v_h))+(\div_h\bs v_h,p)-(\bs f, \bs I_h^{RT}\bs v_h) \lesssim \nu h^{k}|\bs u|_{k+1}\|\bs\varepsilon_h(\bs v_h)\|_{0}.
$$
Finally \eqref{eq:pressrobustenergyerror} holds from \eqref{eq:20180830-5} and the proof of Theorem~\ref{thm:energyerror}.
\end{proof}

The estimate \eqref{eq:pressrobustenergyerror} is pressure-robust in the sense that the right hand side of~\eqref{eq:pressrobustenergyerror} only involves the velocity $\boldsymbol{u}$, no pressure $p$ and $\boldsymbol{f}$.

\begin{remark}\rm
The velocity error in \cite{BeiraodaVeigaLovadinaVacca2018,LiuLiNie2020} depends on a higher order loading effect, thus indirectly depends on the pressure.
Very recently a similar idea, i.e. a modification of the right hand side based on the Raviart-Thomas approximation on a local subtriangulation of the polygons, is applied to derive a pressure-robust conforming virtual element method for Stokes problem in two dimensions in \cite{FrerichsMerdon2020}.
The interpolation operator in \cite{FrerichsMerdon2020} is defined by a local least square problem, which is indeed almost same as $\bs I_K^{RT}$ except \eqref{eq:IKRT3}.
The local energy projector $\bs\Pi^K$ here is based on the local Stokes problem, while the energy projector in \cite{FrerichsMerdon2020} is based on the local Poisson equation. The computable $\bs\Pi^K\boldsymbol{u}_h$ in \cite{FrerichsMerdon2020} is not divergence-free. 
In consideration of small edges encountered in practice with polytopal grids, we refer to \cite{ApelKempfLinkeMerdon2020} for a pressure-robust Crouzeix–Raviart element method for the Stokes equation on anisotropic meshes.
\end{remark}

\section{Reduced Virtual Element Method}

In this section, we study the reduced version of the nonconforming virtual element method~\eqref{eq:mixedvem1}-\eqref{eq:mixedvem2} following the ideas in \cite{BeiraodaVeigaLovadinaVacca2017}.

Since the solution $\bs u_h$ of the discrete method~\eqref{eq:mixedvem1}-\eqref{eq:mixedvem2} is piecewise divergence-free,
it is possible to discretize the velocity in a subspace of $\bs V_h$, such as satisfying the divergence-free constraint.
To this end,
we suggest the local reduced degrees of freedom $\widetilde {\mathcal N}_k(K)$
\begin{align}
(\bs v, \bs q)_F & \quad\forall~\bs q\in\mathbb P_{k-1}(F; \mathbb R^d) \textrm{ on each }  F\in\mathcal F(K), \label{reducedof1}\\
(\bs v, \bs q)_K & \quad\forall~\bs q\in\mathbb G_{k-2}^{\oplus}(K). \label{reducedof2}
\end{align}
And the reduced space of shape functions is given by
$$
\boldsymbol{\widetilde V}_k(K):=\{\bs v\in\boldsymbol{V}_k(K): \div\boldsymbol v\in\mathbb P_{0}(K)\}.
$$
Let the global reduced virtual element space for the velocity
\[
\widetilde{\bs V}_h:=\{\bs v_h\in \bs V_h: \bs v_h|_K\in \bs{\widetilde V}_k(K)\textrm{ for each } K\in\mathcal T_h\},
\]
and the discrete space for the pressure
\[
\widetilde{\mathcal Q}_h:=\{q_h\in L_0^2(\Omega): q_h|_K\in \mathbb P_{0}(K)\textrm{ for each } K\in\mathcal T_h\}.
\]

Applying the integration by parts, it holds for any $\bs v\in\bs{\widetilde V}_k(K)$ and $q\in \mathbb P_{k-1}(K)$
\begin{align*}
(\bs v, \nabla q)_K&=-(\div\bs v, q)_K+(\bs v\cdot\bs n, q)_{\partial K}\\
&=-(\div\bs v, Q_0^Kq)_K+(\bs v\cdot\bs n, q)_{\partial K}=(\bs v\cdot\bs n, q-Q_0^Kq)_{\partial K}.
\end{align*}
Hence for any $\bs v\in\bs{\widetilde V}_k(K)$, we can compute the $L^2$ projection $Q_{k-2}^K\bs v$ as follows:
\begin{align}
(Q_{k-2}^K\bs v, \bs q)_K&=(\bs v, \bs q)_K\qquad\qquad\qquad\forall~\bs q\in \mathbb G_{k-2}^{\oplus}(K),\label{eq:20190831-1}
\\
(Q_{k-2}^K\bs v, \nabla q)_K&=(\bs v\cdot\bs n, q-Q_0^Kq)_{\partial K}\quad\forall~q\in \mathbb P_{k-1}(K).\label{eq:20190831-2}
\end{align}
And for any $\bs\tau\in\mathbb P_{k-1}(K;\mathbb S)$, it follows from the integration by parts
\[
(\bs\varepsilon(\bs v), \bs\tau)_K=-(\bs v, \div\bs\tau)_K + (\bs v, \bs\tau\bs n)_{\partial K}=-(Q_{k-2}^K\bs v, \div\bs\tau)_K + (\bs v, \bs\tau\bs n)_{\partial K}.
\]
As a result, we can compute the $L^2$ projection $\bs Q_{k-1}^K\bs\varepsilon(\bs v)$ for any $\bs v\in\bs{\widetilde V}_k(K)$ as
\begin{equation}\label{eq:20190831-3}
(\bs Q_{k-1}^K\bs\varepsilon(\bs v), \bs\tau)_K=-(Q_{k-2}^K\bs v, \div\bs\tau)_K + (\bs v, \bs\tau\bs n)_{\partial K}\quad\forall~\bs\tau\in\mathbb P_{k-1}(K;\mathbb S).
\end{equation}
Thanks to \eqref{eq:20190831-1}-\eqref{eq:20190831-3}, for any $\bs v\in\bs{\widetilde V}_k(K)$, the local projection $\bs \Pi_k^{K}\bs v$ is computable based on the degrees of freedom~$\widetilde {\mathcal N}_k(K)$~\eqref{reducedof1}-\eqref{reducedof2}.

Thanks to \eqref{eq:H1projlocal2}, it follows $\div(\bs \Pi_k^{K}\bs v)=\div\bs v\in\mathbb P_0(K)$ for any $\bs v\in\bs{\widetilde V}_k(K)$. Therefore $\bs \Pi_k^{K}\bs{\widetilde V}_k(K)=\bs{\widetilde V}_k(K)\cap\mathbb P_k(K; \mathbb R^d)$.

\begin{theorem}\label{thm:equivalence}
Let $(\bs u_h, p_h)\in\bs V_h\times\mathcal Q_h$ be the solution of the divergence-free nonconforming virtual element method~\eqref{eq:mixedvem1}-\eqref{eq:mixedvem2}, and $(\widetilde{\bs u}_h, \widetilde{p}_h)\in\widetilde{\bs V}_h\times\widetilde{\mathcal Q}_h$ be the solution of the reduced nonconforming virtual element method
\begin{align}
\nu a_h(\widetilde{\bs u}_h, \bs v_h) + b_h(\bs v_h, \widetilde{p}_h)& = \langle\bs f, \bs v_h\rangle \qquad \forall~\bs v_h\in \widetilde{\bs V}_h,\label{eq:reducedmixedvem1}\\
b_h(\widetilde{\bs u}_h, q_h)&=0 \qquad\qquad\quad \forall~q_h\in \widetilde{\mathcal Q}_h. \label{eq:reducedmixedvem2}
\end{align}
Then
\begin{equation}\label{eq:equivalence}
\widetilde{\bs u}_h=\bs u_h,\quad \widetilde{p}_h=Q_h^0p_h.
\end{equation}
\end{theorem}
\begin{proof}
Following Section \ref{section:discreteinf-sup} and noting $\widetilde{\mathcal Q}_h\subset \mathcal Q_h$, the reduced virtual element method~\eqref{eq:reducedmixedvem1}-\eqref{eq:reducedmixedvem2} is uniquely solvable.
Thanks to \eqref{eq:mixedvem2}, we have $\div_h\bs u_h=0$ and thus $\bs u_h\in\widetilde{\bs V}_h$. Taking $\bs v_h\in \widetilde{\bs V}_h\subset\bs V_h$, it follows from \eqref{eq:mixedvem1} that
\[
\nu a_h(\bs u_h, \bs v_h) + b_h(\bs v_h, Q_h^0p_h)=\langle\bs f, \bs v_h\rangle.
\]
In other words, $(\bs u_h, Q_h^0p_h)\in \widetilde{\bs V}_h\times\widetilde{\mathcal Q}_h$ satisfies \eqref{eq:reducedmixedvem1} and \eqref{eq:reducedmixedvem2}, which together with the unique solvability of the reduced virtual element method~\eqref{eq:reducedmixedvem1}-\eqref{eq:reducedmixedvem2} indicates \eqref{eq:equivalence}.
\end{proof}

After obtaining $\bs u_h$ and $Q_h^0p_h$ from the reduced virtual element method~\eqref{eq:reducedmixedvem1}-\eqref{eq:reducedmixedvem2},
we can recover the discrete pressure $p_h$ piecewisely. To this end, let $p_h^{\perp}:=p_h-Q_h^0p_h$ and $p_K^{\perp}:=p_h^{\perp}|_K$ for each $K\in\mathcal T_h$.
And define local homogenous spaces
\[
\bs V_{k,0}(K):=\{\bs v\in \bs V_k(K): \bs Q_{\mathbb G_{k-2}^{\oplus}}^K\bs v=\bs0, \textrm{ and } \bs Q_{k-1}^F\bs v=\bs0 \textrm{ for each } F\in\mathcal F(K)\},
\]
\[
\mathcal Q_{k-1,0}(K):=\mathbb P_{k-1}(K)\cap L_0^2(K).
\]
Apparently $\div\bs V_{k,0}(K)\subset\mathcal Q_{k-1,0}(K)$ and $p_K^{\perp}\in\mathcal Q_{k-1,0}(K)$.

It is easy to see that $\div: \bs V_{k,0}(K)\to\mathcal Q_{k-1,0}(K)$ is an injection, which combined with the fact $\dim\bs V_{k,0}(K)=\dim\mathcal Q_{k-1,0}(K)$ indicates $\div: \bs V_{k,0}(K)\to\mathcal Q_{k-1,0}(K)$ is a bijection. 


For any $\bs v\in \bs V_{k,0}(K)$, let $\bs v_h\in\bs V_h$ be defined as
\[
\bs v_h=\begin{cases}
\bs v & \textrm{ in } K, \\
\bs 0 & \textrm{ in } K'\in\mathcal T_h\backslash{K}.
\end{cases}
\]
Then from \eqref{eq:mixedvem1} we get the local problem
\begin{equation}\label{eq:localpPerp}
(\div\bs v, p_K^{\perp})_K = \langle\bs f, \bs v\rangle_K-\nu a_h^K(\bs u_h, \bs v)\quad \forall~\bs v\in \bs V_{k,0}(K).
\end{equation}
Here we have used the fact that $(\div\bs v, Q_0^Kp_h)_K=0$ for any $\bs v\in \bs V_{k,0}(K)$.
The local problem~\eqref{eq:localpPerp} is well-posed due to the bijection $\div: \bs V_{k,0}(K)\to\mathcal Q_{k-1,0}(K)$. 


In summary, we decouple the virtual element method~\eqref{eq:mixedvem1}-\eqref{eq:mixedvem2} in the following way:
\begin{enumerate}[(1)]
\item First solve the reduced virtual element method~\eqref{eq:reducedmixedvem1}-\eqref{eq:reducedmixedvem2} to obtain $(\bs u_h, Q_h^0p_h)$;
\item then solve the local problem~\eqref{eq:localpPerp} piecewisely to get $p_h^{\perp}$;
\item finally set $p_h=p_h^{\perp}+Q_h^0p_h$.
\end{enumerate}

\section{Numerical Examples}

In this section, some numerical results of the nonconforming virtual element method~\eqref{eq:mixedvem1}-\eqref{eq:mixedvem2} are provided to verify Theorem~\ref{thm:energyerror}, Theorem~\ref{thm:pressrobustenergyerror} and Theorem~\ref{thm:equivalence}.
Let the viscosity $\nu=1$ and $k=2$.
All of the numerical examples are implemented by using the FEALPy package \cite{fealpy}.

\begin{example}\label{ex:1}
Consider the Stokes problem~\eqref{eq:stokes} on the rectangular domain
$\Omega = (0, 1) \times (0, 1)$. Take $\bs f=(0, \mathrm{Ra}\,(1-y+3y^2))^{\intercal}$ with parameter $\mathrm{Ra}> 0$. The exact solution is (cf. \cite[Example~1.1]{JohnLinkeMerdonNeilanEtAl2017})
    \begin{equation*}
\bs u = \boldsymbol{0}, \quad   p = \mathrm{Ra}\,(y^3 - y^2/2 + y - 7/12).
    \end{equation*}
The parameter $\mathrm{Ra}$ only affects the pressure.
\end{example}

The rectangular domain $\Omega$ is partitioned by the uniform triangle mesh.
The numerical results of error $\| \boldsymbol \varepsilon(\boldsymbol u) -
\boldsymbol \varepsilon_h(\boldsymbol \Pi_h \boldsymbol u_h) \|_0$ with
$\text{Ra} = 1, 10^2, 10^4, 10^6$ for the virtual element
method~\eqref{eq:mixedvem1}-\eqref{eq:mixedvem2} and the pressure-robust virtual
element method~\eqref{eq:mixedvempressrobust1}-\eqref{eq:mixedvempressrobust2}
are listed in Figure~\ref{fig:ex1strian}.  From the left subfigure in
Figure~\ref{fig:ex1strian}, we observe that $\|\boldsymbol
\varepsilon(\boldsymbol u) - \boldsymbol \varepsilon_h(\boldsymbol \Pi_h
\boldsymbol u_h) \|_0$ achieves the optimal convergence rate $O(h^{2})$ for the
virtual element method~\eqref{eq:mixedvem1}-\eqref{eq:mixedvem2}, which is in
coincidence with Theorem~\ref{thm:energyerror}, but not pressure-robust.  And we
can see from the right subfigure in Figure~\ref{fig:ex1strian} that
$\|\boldsymbol \varepsilon(\boldsymbol u) - \boldsymbol
\varepsilon_h(\boldsymbol \Pi_h \boldsymbol u_h) \|_0$ for the virtual element
method~\eqref{eq:mixedvempressrobust1}-\eqref{eq:mixedvempressrobust2} is zero
up to round-off errors, as indicated by
Theorem~\ref{thm:pressrobustenergyerror}. Hence the virtual element
method~\eqref{eq:mixedvempressrobust1}-\eqref{eq:mixedvempressrobust2} is
pressure-robust.
\begin{figure}[htbp]
\centering
\subfigure[VEM~\eqref{eq:mixedvem1}-\eqref{eq:mixedvem2}.]{
\begin{minipage}[t]{0.49\linewidth}
\centering
\includegraphics[width=6cm]{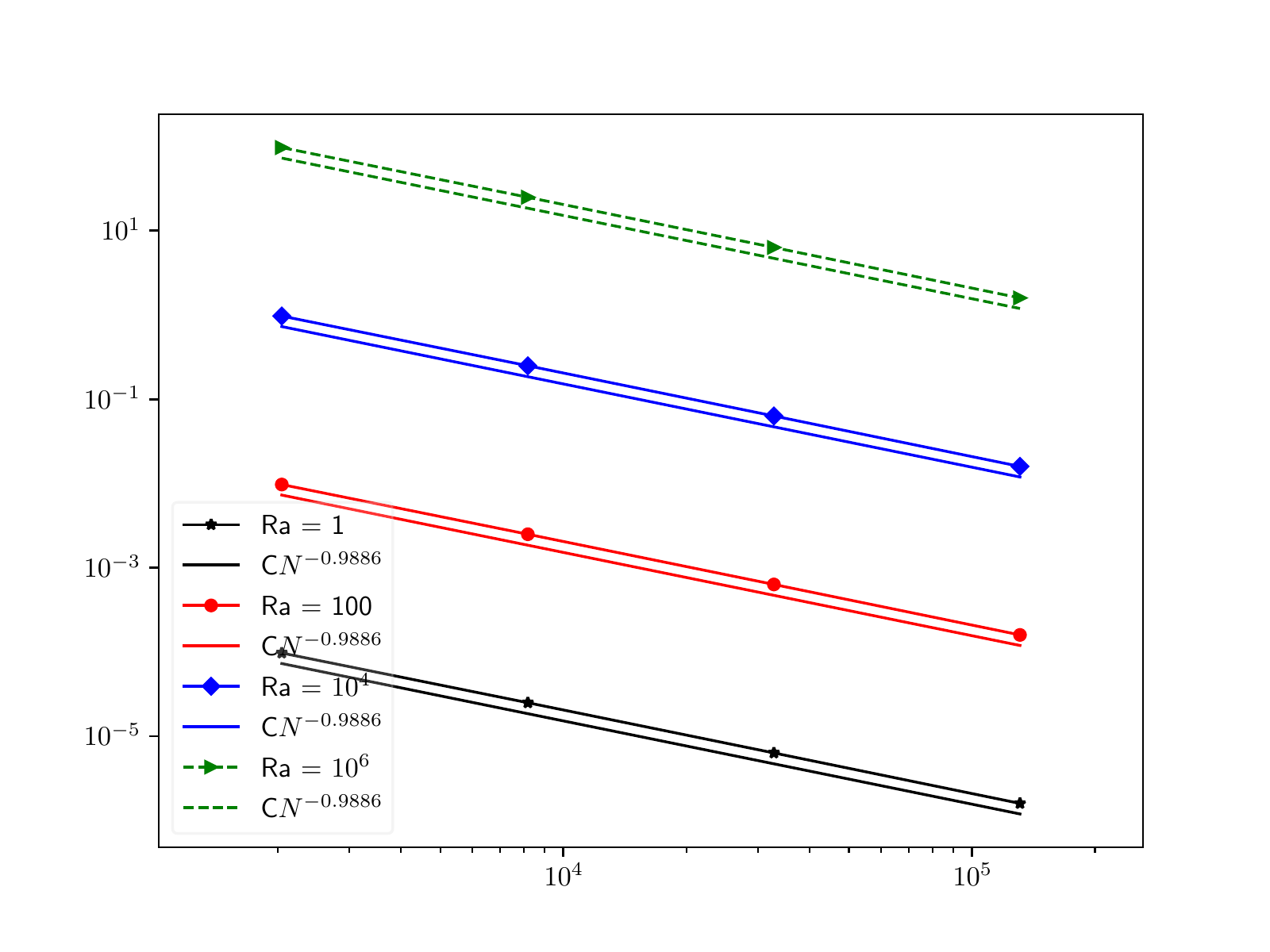}
\end{minipage}%
}%
\subfigure[Pressure-robust VEM~\eqref{eq:mixedvempressrobust1}-\eqref{eq:mixedvempressrobust2}.]{
\begin{minipage}[t]{0.49\linewidth}
\centering
\includegraphics[width=6cm]{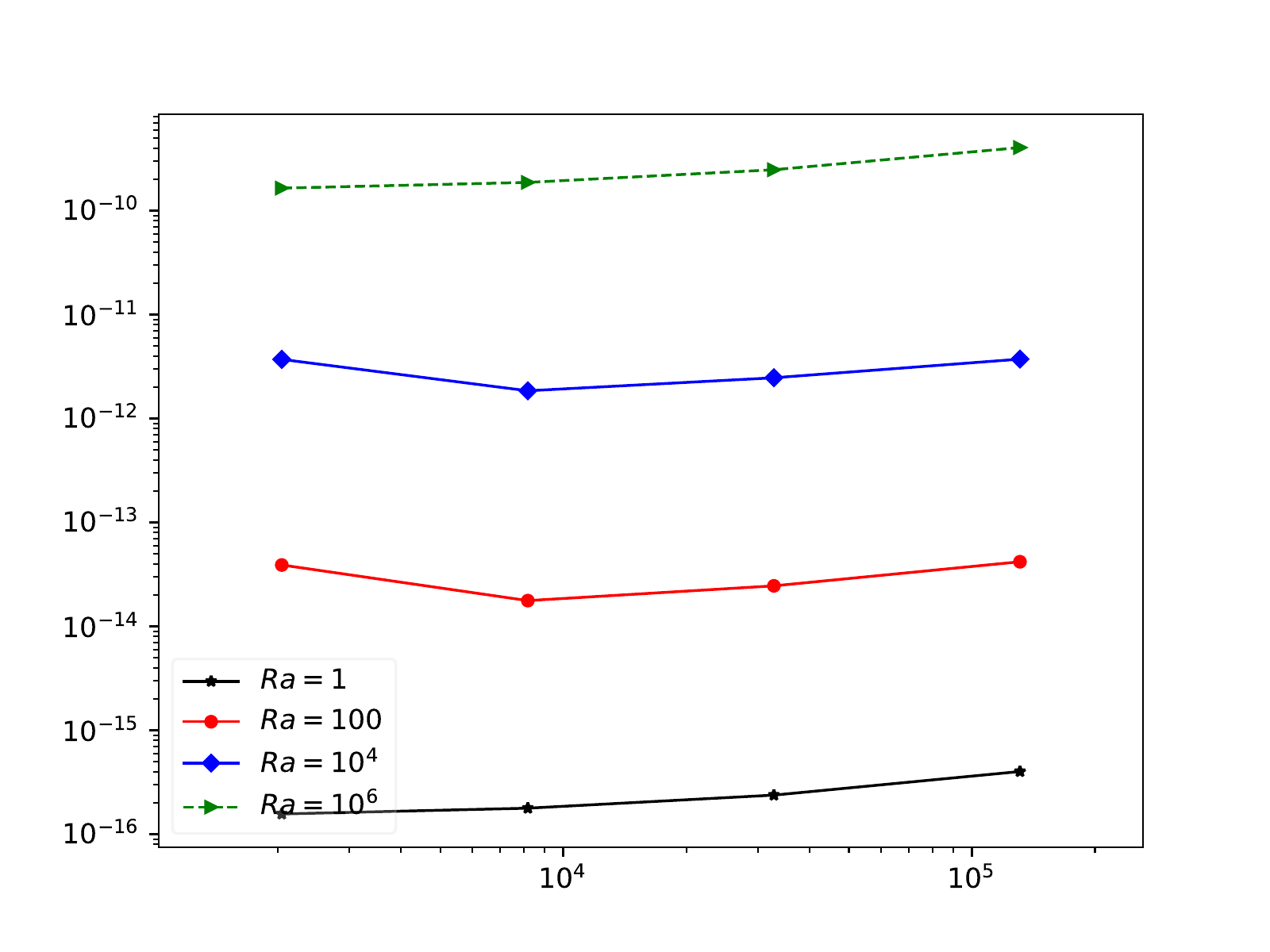}
\end{minipage}%
}%
\centering
\caption{Error $\| \boldsymbol \varepsilon(\boldsymbol u) - \boldsymbol \varepsilon_h(\boldsymbol \Pi_h \boldsymbol u_h) \|_0$ of Example~\ref{ex:1} with $k=2$.}
\label{fig:ex1strian}
\end{figure}

\begin{example}\label{ex:5}
   Consider the Stokes problem~\eqref{eq:stokes} on the L-shaped domain $\Omega=(-1, 1) \times (-1, 1)
    \setminus[0, 1) \times (-1, 0]$. The exact solution is taken as
     \begin{equation*}
        \begin{aligned}
            \bs u & = (2(x^3 - x)^2(3y^2 - 1)(y^3 - y), 
            (3x^2 - 1)(-2x^3 + 2x)(y^3 -y)^2)^{\intercal}, \\
            p & = \frac{1}{x^2 +1} - \frac{\pi}{4}.
        \end{aligned}
    \end{equation*}
\end{example}

The exact solution $(\bs u, p)$ is smooth although the L-shaped domain $\Omega$
is nonconvex.  We present the polygonal mesh and the corresponding numerical
velocity flow with $k=2$ in Figure~\ref{fig:ex5mesh}.  By the numerical results
in Table \ref{tablex5p2hexagon},  we can see that $\| p - \tilde p_h\|_0=O(h)$,
$\| \boldsymbol \varepsilon(\boldsymbol u) - \boldsymbol
\varepsilon_h(\boldsymbol \Pi_h \tilde{\boldsymbol u}_h) \|_0=O(h^{2})$ and $\|
p - p_h\|_0=O(h^{2})$, which coincide with the theoretical error estimates in
Theorem~\ref{thm:energyerror} and Theorem~\ref{thm:equivalence}.  The
convergence rates of $\|\boldsymbol u - \boldsymbol \Pi_h\boldsymbol u_h\|_0=\|
\boldsymbol u - \boldsymbol \Pi_h \tilde{\boldsymbol u}_h\|_0=O(h^{3})$ are
higher than the optimal ones on the L-shaped domain, which is probably caused by
the uniform meshes. To make the article more concise, here we only show the
numerical results of $k = 2$. For $ k> 2$, one can run the test
script, named {\em StokesRDFNCVEM2d\_example.py}, in directory of {\em FEALPy/example} \cite{fealpy}.


\begin{figure}[htbp]
\centering
\subfigure[Polygonal mesh.]{
\begin{minipage}[t]{0.49\linewidth}
\centering
\includegraphics[width=6cm]{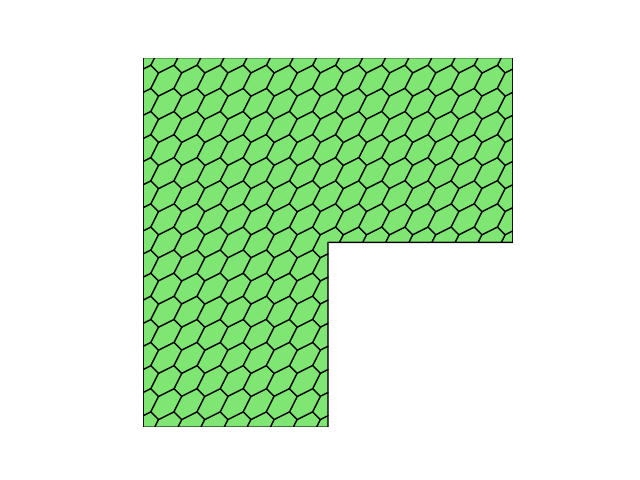}
\end{minipage}%
}%
\subfigure[Numerical velocity.]{
\begin{minipage}[t]{0.49\linewidth}
\centering
\includegraphics[width=6cm]{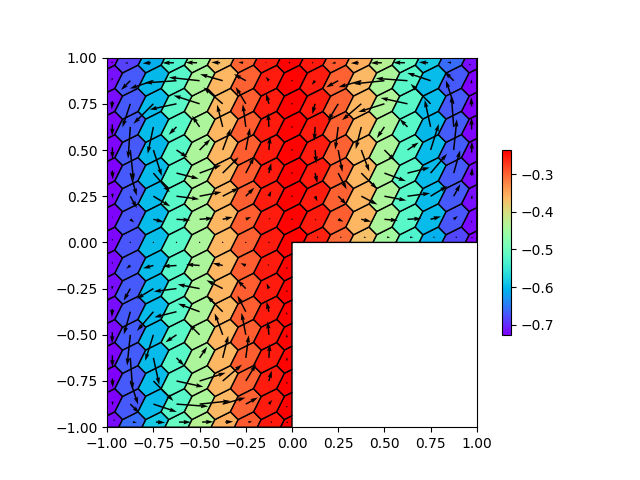}
\end{minipage}%
}%
\centering
\caption{Mesh for L-shaped domain and numerical velocity of Example~\ref{ex:5} with $k=2$.
}
\label{fig:ex5mesh}
\end{figure}


\begin{table}[htbp]
\begin{center}
{\small
    \caption{Numerical results for Example~\ref{ex:5} with $k=2$.}
\begin{tabular}[c]{|c|c|c|c|c|}
\hline
$\#\mathcal T_h$ &   65 &  225 &  833 & 3201 \\
\hline
$\| \boldsymbol u - \boldsymbol \Pi_h\tilde{\boldsymbol u}_h\|_0$ & 3.5827e-03 & 6.6167e-04 & 9.2871e-05 & 1.2026e-05
\\\hline
Order & -- & 2.44 & 2.83 & 2.95
\\\hline
$\| p - \tilde p_h\|_0$ & 6.5700e-02 & 3.3869e-02 & 1.7176e-02 & 8.6490e-03
\\\hline
Order & -- & 0.96 & 0.98 & 0.99
\\\hline
$\| \boldsymbol \varepsilon(\boldsymbol u) - \boldsymbol \varepsilon_h(\boldsymbol \Pi_h \tilde{\boldsymbol u}_h) \|_0$ & 6.7184e-02 & 2.3092e-02 & 6.8821e-03 & 1.8805e-03
\\\hline
Order & -- & 1.54 & 1.75 & 1.87
\\\hline
$\|\boldsymbol u - \boldsymbol \Pi_h\boldsymbol u_h\|_0$ & 3.5827e-03 & 6.6167e-04 & 9.2871e-05 & 1.2026e-05
\\\hline
Order & -- & 2.44 & 2.83 & 2.95
\\\hline
$\| p - p_h\|_0$ & 1.4686e-02 & 4.2207e-03 & 1.0318e-03 & 2.2019e-04
\\\hline
Order & -- & 1.8  & 2.03 & 2.23
\\\hline
\end{tabular}\label{tablex5p2hexagon}
}
\end{center}
\end{table}

\bibliographystyle{siamplain}
\bibliography{paper}
\end{document}

%% file: ex_shared.tex
\usepackage{lipsum}
\usepackage{amsfonts}
\usepackage{graphicx}
\usepackage{epstopdf}
\usepackage{algorithmic}
\ifpdf
  \DeclareGraphicsExtensions{.eps,.pdf,.png,.jpg}
\else
  \DeclareGraphicsExtensions{.eps}
\fi

\newsiamremark{remark}{Remark}
\newsiamremark{hypothesis}{Hypothesis}
\crefname{hypothesis}{Hypothesis}{Hypotheses}
\newsiamthm{claim}{Claim}

\headers{Divergence-Free Nonconforming VEM For Stokes}{H. Wei, X. Huang, and A. Li}

\title{Piecewise Divergence-Free Nonconforming Virtual Elements for Stokes Problem in Any Dimensions\thanks{
The first and third authors were supported by NSFC (11871413) and in part by projects
of Education Department of Hunan Provincial of China (19B534, 19A500). The second author was supported by the NSFC (11771338, 12071289), and the Fundamental Research Funds for the Central
Universities (2019110066).}}

\author{Huayi Wei\thanks{Hunan Key Laboratory for Computation and Simulation in Science and Engineering; School of Mathematics and Computational Science,
Xiangtan University, Xiangtan 411105, P.R.China 
  (\email{weihuayi@xtu.edu.cn}, \email{201721511166@smail.xtu.edu.cn}).}
\and 
Xuehai Huang\thanks{Corresponding author. School of Mathematics, Shanghai University of Finance and Economics, Shanghai 200433, China 
  (\email{huang.xuehai@sufe.edu.cn}).}
\and 
Ao Li\footnotemark[2]}

\usepackage{amsopn}
